\documentclass[11pt]{article}

\usepackage[margin=1in]{geometry}
\usepackage{graphicx}

\usepackage{mathtools}
\usepackage{amssymb}
\usepackage{amsthm}
\usepackage{thmtools}
\usepackage{mleftright}
\usepackage{fourier}

\usepackage{enumitem}
\usepackage{titlesec}
\usepackage[dvipsnames]{xcolor}

\usepackage[pagebackref]{hyperref}
\usepackage[nameinlink,capitalise]{cleveref}
\usepackage{doi}

\hypersetup{
  colorlinks=true,
  linkcolor=Green,
  citecolor=NavyBlue,
  filecolor=Plum,
  urlcolor=Plum
}

\renewcommand*{\backref}[1]{}
\renewcommand*{\backrefalt}[4]{%
  \ifcase #1
    (No citations.)%
  \or
    (Cited page~#2.)%
  \else
    (Cited pages~#2.)%
  \fi
}

\newcommand{\real}{\mathbb{R}}
\newcommand{\complex}{\mathbb{C}}
\newcommand{\field}{\mathbb{K}}

\DeclareMathOperator{\tr}{tr}
\DeclareMathOperator{\diag}{diag}
\DeclareMathOperator{\range}{range}
\DeclareMathOperator{\rank}{rank}
\newcommand{\mat}[1]{\boldsymbol{#1}}
\renewcommand{\vec}[1]{\boldsymbol{#1}}
\newcommand{\lowrank}[1]{%
  \mleft[\!\mleft[ #1 \mright]\!\mright]}
\newcommand{\norm}[1]{\mleft\| #1 \mright\|}
\newcommand{\Id}{\mathbf{I}}
\newcommand{\evec}{\mathbf{e}}

\newcommand{\twobytwo}[4]{%
  \begin{bmatrix}#1 & #2 \\ #3 & #4\end{bmatrix}}
\newcommand{\twobyone}[2]{%
  \begin{bmatrix}#1 \\ #2\end{bmatrix}}
\newcommand{\onebytwo}[2]{%
  \begin{bmatrix}#1 & #2\end{bmatrix}}

\DeclareMathOperator{\expect}{\mathbb{E}}
\DeclareMathOperator{\Unif}{\textsc{Unif}}
\newcommand{\order}{\mathcal{O}}

\DeclareMathOperator*{\argmin}{argmin}
\newcommand{\iu}{\mathrm{i}}

\renewcommand{\tilde}[1]{\widetilde{#1}}

\newcommand{\Wishart}{\textnormal{\textsc{Wishart}}}
\newcommand{\Beta}{\textnormal{\textsc{Beta}}}
\newcommand{\hatbold}[1]{%
  \skew{4}\widehat{\smash{\boldsymbol{#1}}\mathstrut}}
\newcommand{\Ahat}{\smash{\hatbold{A}}}
\newcommand{\Xhat}{\smash{\hatbold{X}}}
\newcommand{\Hhat}{\smash{\hatbold{H}}}

\makeatletter
\def\th@plain{%
  \thm@notefont{}%
  \itshape
}
\def\th@definition{%
  \thm@notefont{}%
  \normalfont
}
\makeatother

\declaretheorem[name=Theorem,numberwithin=section]{theorem}
\declaretheorem[name=Proposition,numberlike=theorem]{proposition}
\declaretheorem[name=Fact,numberlike=theorem]{fact}
\declaretheorem[name=Conjecture,numberlike=theorem]{conjecture}
\declaretheorem[name=Lemma,numberlike=theorem]{lemma}

\theoremstyle{definition}
\declaretheorem[name=Definition,numberlike=theorem]{definition}

\numberwithin{equation}{section}

\crefname{equation}{}{}
\crefname{section}{section}{sections}
\crefname{appendix}{appendix}{appendices}

\titleformat{\subsection}[runin]
  {\normalfont\normalsize\bfseries}
  {\thesubsection}{0.5em}{}[\textbf{.}]

\titleformat{\subsubsection}[runin]
  {\normalfont\normalsize\scshape}
  {\thesubsubsection}{0.5em}{}[.]

\titleformat{\paragraph}[runin]
  {\normalfont\normalsize\bfseries\itshape}
  {\textit{\theparagraph}}{0.5em}{}[]

\newcommand{\mytitle}{%
  Sharp analysis of sketched least squares and randomized low-rank approximation}

\title{\mytitle}

\author{%
  Ethan N. Epperly%
  \thanks{Department of Mathematics, University of California Berkeley
  (\href{mailto:eepperly@berkeley.edu}{eepperly@berkeley.edu},
  \url{https://ethanepperly.com})}
  \and
  Robert J. Webber%
  \thanks{Department of Mathematics, University of California San Diego
  (\href{mailto:rwebber@ucsd.edu}{rwebber@ucsd.edu},
  \url{https://sites.google.com/ucsd.edu/rwebber/})}}

\date{\today}

\begin{document}

\maketitle

\begin{abstract}
Two widely used randomized algorithms are the sketch-and-solve method for least-squares regression and the randomized SVD for low-rank approximation.
These algorithms apply a random embedding to compress a target matrix, and they perform computations on the compressed matrix to save computational cost.
This paper asks, what is the optimal random embedding in these algorithms?
Also, what is the sharpest possible error bound for the optimal embedding?
The paper proves that a random orthonormal matrix is minimax optimal for the sketch-and-solve algorithm while any full-rank rotation-invariant embedding is minimax optimal for the randomized SVD.
Following these results, the paper obtains the best possible error bounds for sketched least-squares and the randomized SVD.
Last, empirical experiments provide evidence of universality phenomena, in which several random embeddings lead to similar accuracy to the optimal embeddings in practice.
\end{abstract}

\section{Introduction}

Two cornerstone algorithms in randomized linear algebra are the sketch-and-solve method for linear least-squares regression \cite{DMM06,Sar06} and the randomized SVD \cite{HMT11} for low-rank approximation.
These algorithms multiply a high-dimensional input matrix $\mat{A}$ on the right or left by a random embedding $\mat{\varOmega}$ and process the low-dimensional output $\mat{A} \mat{\varOmega}$ or $\mat{\varOmega}^*\mat{A}$ to produce an approximate solution to a linear algebra problem.
It is believed that the most accurate, if not the most computationally efficient, choice for the random embedding $\mat{\varOmega}$ is a \emph{random orthonormal embedding}, a uniformly random matrix with orthonormal columns.

This paper confirms the prevailing belief.
We analyze the mean-squared error of sketch-and-solve and the randomized SVD with worst-case input matrices, and we identify \emph{minimax optimal} embeddings that minimize the worst-case mean-squared error.
For the sketch-and-solve algorithm, we prove that a random orthonormal embedding is minimax optimal.
For the randomized SVD, we establish the minimax optimality of any full-rank rotation-invariant random embedding.
In the process, we derive sharp error bounds on the mean-squared error of the sketch-and-solve and randomized SVD algorithms that depend only on the rank of $\mat{A}$ and dimensions of $\mat{\varOmega}$.
These bounds improve on all past work, and they are unimprovable without assuming more structure on $\mat{A}$.
The bounds determine the minimal dimensions needed to match a given error threshold, and we identify hard instances that fully saturate the error bounds.

Our work contributes to an ongoing research effort \cite{DL19a,TYUC19,MT20,Der23,DM23,MDM+23a,CEMT25,Epp25a} to understand which random embeddings are most accurate and computationally efficient
in randomized linear algebra.
We identify sharp limits on the behavior of any random embedding and establish that a random orthonormal matrix meets these limits.
Thus, random orthonormal matrices can serve as a standard of comparison for other, more computationally efficient random embeddings.

Last, we perform numerical experiments with several embeddings in common use.
Consistent with past work \cite{DL19a,TYUC19,MT20,Der23,CEMT25}, we observe that sketch-and-solve and the randomized SVD exhibit universal behavior: for these algorithms, any well-designed random embedding performs similarly to either a Gaussian embedding or a random orthonormal embedding with the same dimensions.
This paper provides empirical evidence of universality and discusses when embeddings from the Gaussian and random orthonormal universality classes behave differently.
Looking forward, further theoretical work is needed to validate our observations of universality.

\subsection{Definitions} \label{sec:definitions}
We work over the field $\field$ of real or complex numbers.
To state results over both fields, we define
\begin{equation*}
    \alpha_\field = \begin{cases}
            1, & \text{if } \field = \real, \\
            0, & \text{if } \field = \complex.
    \end{cases}
\end{equation*}
The adjoint is ${}^*$ and the Moore--Penrose pseudoinverse is ${}^+$.
For this paper, a positive-semidefinite (psd) matrix is conjugate symmetric and has nonnegative eigenvalues.
We use $\mat{\varPi}_{\mat{F}}$ to denote the orthogonal projection onto the range of $\mat{F}$.

The real standard normal distribution is denoted $\mathcal{N}_\real(0,1)$.
The complex standard normal distribution is denoted $\mathcal{N}_\complex(0,1)$, and a draw from this distribution may be generated as $(g_1 + \iu g_2) / \sqrt{2}$ for independent $g_1,g_2 \sim \mathcal{N}_\real(0,1)$.
An embedding $\mat{\varOmega} \in \field^{n \times \ell}$ is any tall matrix with $n \geq \ell$.
A random embedding $\mat{\varOmega} \in \field^{n \times \ell}$ is rotation-invariant if $\mat{U}\mat{\varOmega} \sim \mat{\varOmega}$ for every unitary matrix $\mat{U} \in \field^{n\times n}$.
In particular, a rotation-invariant embedding $\mat{\varOmega} \in \field^{n\times \ell}$ is called a \emph{random orthonormal matrix} if $\mat{\varOmega}^*\mat{\varOmega}= \Id$ with probability one.
It is called a \emph{Gaussian embedding} over $\field$ if it is populated with independent $\mathcal{N}_\field(0,1)$ random variables.

\subsection{Organization}

The rest of the paper is organized as follows.
Theoretical analysis and universality results for sketch-and-solve are in
\cref{sec:sketch_and_solve} and matching results for the randomized SVD are in \cref{sec:rsvd-description}.
Next, 
\cref{sec:gen_nystrom} provides a complementary analysis of the generalized Nystr\"om approximation, which may be of independent interest.
The proofs are divided into two sections: \cref{sec:sketch_analysis} analyzes sketch-and-solve and \cref{sec:rsvd-nystom-analysis} analyzes low-rank approximation.
\Cref{sec:conclusion} concludes and presents open problems.

\section{Results for sketch-and-solve} \label{sec:sketch_and_solve}

The sketch-and-solve method was developed by Drineas, Mahoney, \& Muthukrishnan \cite{DMM06} and Sarl\'os \cite{Sar06} to approximate the solution to a linear least-squares problem
\begin{equation} \label{eq:least-squares}
    \min_{\mat{X} \in \field^{d\times p}} \norm{\mat{B} - \mat{A}\mat{X}}_{\rm F}^2 \quad \text{where } \mat{A} \in \field^{n\times d} \text{ and } \mat{B} \in \field^{n\times p}.
\end{equation}
The minimum-norm solution of this problem is $\mat{X}_\star = \mat{A}^+\mat{B}$.
To approximate the minimum-norm solution, the sketch-and-solve method multiplies $\mat{A}$ and $\mat{B}$ by a random embedding $\mat{\varOmega} \in \field^{n \times \ell}$ on the left and solves the ``sketched'' least-squares problem
\begin{equation} \label{eq:sketch-solve}
    \Xhat \coloneqq
    (\mat{\varOmega}^*\mat{A})^+ (\mat{\varOmega}^*\mat{B}) \in \argmin_{\mat{X} \in \field^{d\times p}} \norm{\smash{\mat{\varOmega}^*\mat{B} - (\mat{\varOmega}^*\mat{A})\mat{X}}}_{\rm F}^2.
\end{equation}
The resulting approximation $\Xhat$ is called the \emph{sketch-and-solve solution}.

\subsection{Sketch-and-solve: Theoretical analysis} \label{sec:sketch-solve-results}

We provide an exact formula for the sketch-and-solve residual norm when the embedding is a Gaussian or a random orthonormal matrix.
The proof is in \cref{sec:partial-isometry-proof}.
\begin{theorem}[Sketch-and-solve: Exact analysis] \label{thm:partial-isometry}
    Let $\mat{\varOmega} \in \field^{n \times \ell}$ be a Gaussian or random orthonormal embedding, and consider matrices $\mat{A} \in \field^{n\times d}$ and $\mat{B} \in \field^{n\times p}$, for which $r \coloneqq \rank(\mat{A}) < \ell - \alpha_\field$.    
    Then the sketch-and-solve solution $\Xhat = (\mat{\varOmega}^*\mat{A})^+ (\mat{\varOmega}^*\mat{B})$ satisfies 
        \begin{equation*}
        \begin{alignedat}{2}
        \expect[\Xhat] &= \mat{A}^+ \mat{B},& \quad \text{(either embedding)}
        \end{alignedat}
    \end{equation*}
    and either
    \begin{equation*}
        \begin{alignedat}{2}
        \expect \norm{\smash{\mat{B} - \mat{A}\Xhat}}_{\rm F}^2 &= \Biggl(1 + \frac{r}{\ell - r - \alpha_\field}\Biggr) \min_{\mat{X}\in\field^{d\times p}}\norm{\mat{B} - \mat{A}\mat{X}}_{\rm F}^2, & \quad \text{(Gaussian)} \\
        \text{or } \expect \norm{\smash{\mat{B} - \mat{A}\Xhat}}_{\rm F}^2 &= \Biggl(1 + \frac{n-\ell}{n-r} \frac{r}{\ell - r - \alpha_\field}\Biggr) \min_{\mat{X}\in\field^{d\times p}}\norm{\mat{B} - \mat{A}\mat{X}}_{\rm F}^2. & \quad \text{(random orthonormal)}
        \end{alignedat}
    \end{equation*}
\end{theorem}
\noindent The result for real Gaussian embeddings was proved by Bartan \& Pilanci \cite[Lem.~2.1]{BP20}, and we extend their result to complex Gaussians.
We believe the results for real and complex random orthonormal matrices are new.
\Cref{thm:partial-isometry} leads to the following practical implications.
\begin{itemize}
    \item \textbf{\textit{Orthogonality helps, but only a bit.}}
    Random orthonormal embeddings lead to less error than Gaussian embeddings.
    However, the practical setting uses $\ell\ll n$, so the difference between the embeddings is small.
    There is little reason to prefer the random orthonormal matrices over Gaussian matrices for most use cases.
    \item \textit{\textbf{Complex numbers help, but only a bit.}}
    The difference between real and complex embeddings is tiny.
    A real Gaussian or random orthonormal embedding of dimension $\ell$ yields nearly the same accuracy as its complex analog of dimension $\ell-1$.
    \item \textit{\textbf{Minimal embedding dimension.}}
    When the matrix in the least-squares problem has rank $r$, the minimal embedding dimension to achieve a finite expected error is $\ell_{\rm min} = r + 1 + \alpha_\field$.
    A Gaussian embedding with the minimal dimension achieves
    \begin{equation*}
        \expect \norm{\smash{\mat{B} - \mat{A}\Xhat}}_{\rm F}^2
        = (r + 1) \min_{\mat{X}\in\field^{d\times p}}\norm{\mat{B} - \mat{A}\mat{X}}_{\rm F}^2.
    \end{equation*}
    A random orthonormal embedding with the minimal dimension achieves a slightly better mean-squared error.
    \item \textit{\textbf{Sufficient embedding dimension.}}
    An embedding dimension of $\ell \ge r/\varepsilon + r + \alpha_\field$ suffices to guarantee the following $\varepsilon$-accuracy condition:
    \begin{equation*}
        \expect \norm{\smash{\mat{B} - \mat{A}\Xhat}}_{\rm F}^2 \le (1+\varepsilon) \min_{\mat{X}\in\field^{d\times p}}\norm{\mat{B} - \mat{A}\mat{X}}_{\rm F}^2.
    \end{equation*}
    A simpler analysis (e.g., \cite[Prop.~5.3]{KT24}) gives the incorrect scaling $\order(r/\varepsilon^2)$.
\end{itemize}
\Cref{thm:partial-isometry} justifies these statements rigorously for Gaussian and random orthonormal matrices.
By universality, we expect these conclusions to extend to more general random embeddings.

Our second result establishes a sharp lower bound on the sketch-and-solve error for any random embedding $\mat{\varOmega}$ that does not depend on the problem data $(\mat{A},\mat{B})$.
The proof is in \cref{sec:proof_optimality}.
\begin{theorem}[Sketch-and-solve: Sharp lower bound] \label{thm:sketch-solve-optimality}
    Let $\mat{\varOmega} \in \field^{n \times \ell}$ be any full-rank random embedding.
    For any rank parameter $r \in \mathbb{N}$ with $r < \ell - \alpha_\field$ and any problem dimensions $d, p \in \mathbb{N}$ with $d \geq r$,
    \begin{equation*}
        \sup_{\substack{\mat{A}\in \field^{n\times d},\: \mat{B} \in \field^{n\times p} \\
        \rank(\mat{A}) = r,\: \mat{B} \neq \mat{A} \mat{A}^+ \mat{B}}} \frac{\expect \norm{\mat{B} - \mat{A}(\mat{\varOmega}^*\mat{A})^+ (\mat{\varOmega}^*\mat{B})}_{\rm F}^2}{\norm{\mat{B} - \mat{A} \mat{A}^+ \mat{B}}_{\rm F}^2} 
        \ge 1 + \frac{n-\ell}{n-r}\, \frac{r}{\ell - r - \alpha_\field}.
    \end{equation*}
\end{theorem}

\noindent The combination of \cref{thm:partial-isometry} with \cref{thm:sketch-solve-optimality} establishes that a random orthonormal matrix is the minimax optimal random embedding for sketch-and-solve.
Here, we understand minimax optimality in the standard game-theoretic sense.
Player 1 (the algorithm designer) chooses a random distribution for the matrix $\mat{\varOmega}$.
Then Player 2 (the adversary) uses their knowledge of the distribution of $\mat{\varOmega}$ to select a challenging inconsistent least-squares problem described by a pair $(\mat{A},\mat{B})$.
Player 1 seeks to minimize $\expect \norm{\smash{\mat{B} - \mat{A}(\mat{\varOmega}^*\mat{A})^+ (\mat{\varOmega}^*\mat{B})}}_{\rm F}^2 / \norm{\mat{B} - \mat{A} \mat{A}^+ \mat{B}}_{\rm F}^2$, while Player 2 seeks to maximize this ratio.
Together, \cref{thm:partial-isometry,thm:sketch-solve-optimality} show that Player 1's optimal strategy is to choose $\mat{\varOmega}$ as a random orthonormal embedding.

\subsection{Sketch-and-solve: Empirical evidence of universality} \label{sec:sketch-solve-discussion}

In most applications of the sketch-and-solve method, we require a fast algorithm for computing $\mat{\varOmega}^*\mat{A}$ \cite[sec.~9]{MT20}, which is not available for the Gaussian or random orthonormal embedding. 
Thus we cannot usually invoke \cref{thm:partial-isometry} directly.
Nevertheless, many common random embeddings perform similarly to a Gaussian or random orthonormal embedding, falling into one of two universality classes \cite{DL19a,TYUC19,Der23}.
\begin{itemize}
    \item The \textbf{\textit{Gaussian universality class}} contains embeddings whose behavior is nearly Gaussian. 
    Embeddings with nearly independent and identically distributed (iid) entries tend to fall into this class.
    \item The \textbf{\textit{random orthonormal universality class}} contains embeddings whose behavior is similar to a random orthonormal embedding.
    Random embeddings with orthonormal columns usually fall into this class.
\end{itemize}
\begin{figure}
    \centering
    \includegraphics[width=0.99\linewidth]{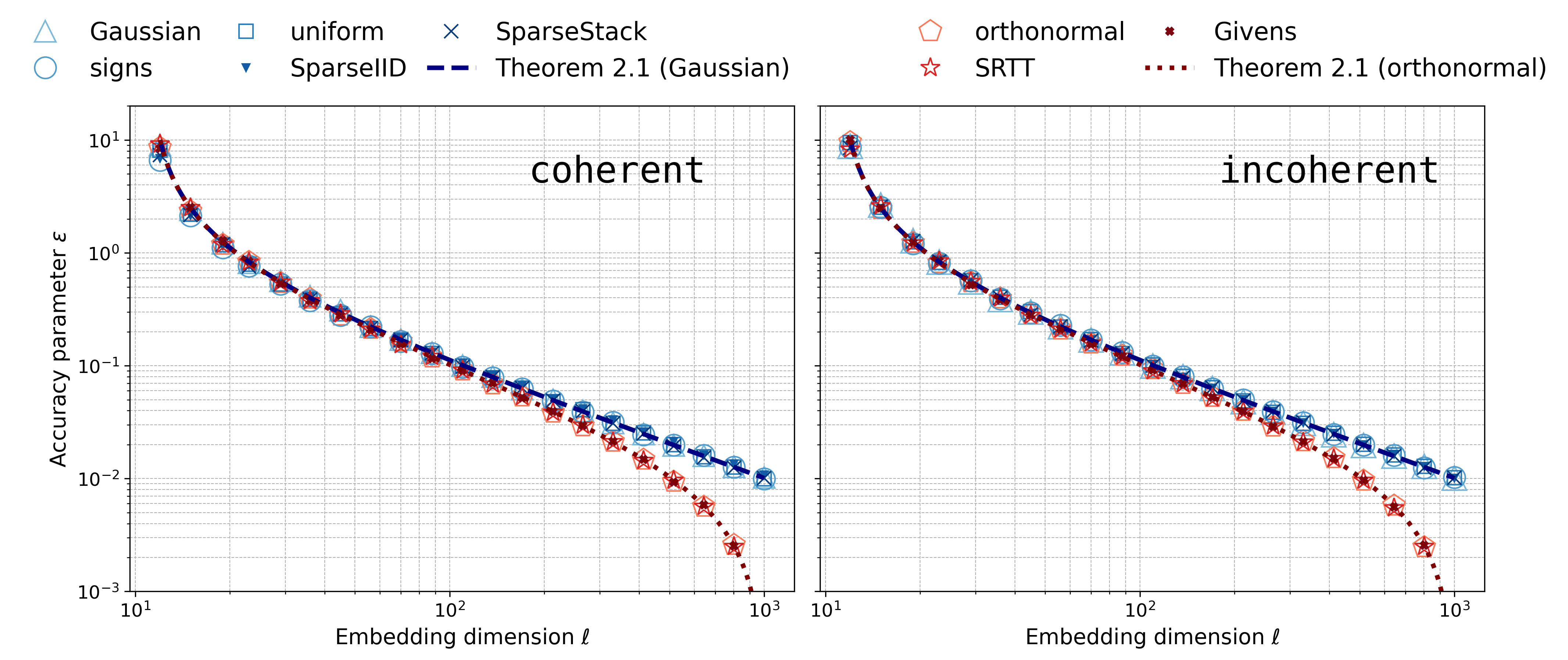}
    \caption{Accuracy parameter $\varepsilon = \expect \norm{\mat{B} - \mat{A}(\mat{\varOmega}^*\mat{A})^+ (\mat{\varOmega}^*\mat{B})}_{\rm F}^2 / \norm{\mat{B} - \mat{A} \mat{A}^+ \mat{B}}_{\rm F}^2 - 1$ as a function of embedding dimension $\ell$ for eight random embeddings.
    All embeddings follow the theoretical predictions of \cref{thm:partial-isometry} for either the Gaussian matrix or random orthonormal matrix.}
    \label{fig:universality}
\end{figure}
\Cref{thm:partial-isometry} leads to precise predictions for the two universality classes.
\Cref{fig:universality} illustrates these predictions on two separate problem instances: $\mat{A} = [\begin{matrix}\Id &\mat{0}\end{matrix}]^*$ (\emph{left}, ``coherent'') and $\mat{A}$ consisting of the first $d$ columns of the discrete cosine transform (\emph{right}, ``incoherent'').
For each problem instance, we set the dimensions $n = 10^3$, $d = 10^1$, $p = 1$, and we choose $\mat{B} = \sum_{i=1}^n i\evec_i$.
We test eight random embeddings across a range of embedding dimensions $\ell \in \{d+2, \ldots, n-1\}$.
\begin{itemize}
    \item In the \textbf{\textit{Gaussian universality class}}, we test three embeddings with iid entries including Gaussian $\mathcal{N}_\real(0,1)$, random sign $\Unif\{\pm 1\}$, and random uniform $\Unif [-1,1]$ distributions.
    We also test SparseIID and SparseStack random matrices, which have an expected value of $\zeta$ nonzero entries per column; see \cite{CEMT25} for definitions.
    Following the universality experiments in \cite[App.~B]{Epp25a}, we set $\zeta = 16$ for the SparseIID construction and set $\zeta = 8$ for the more effective SparseStack construction.
    \item In the \textbf{\textit{random orthonormal matrix universality class}}, we test three embeddings with orthonormal columns including the random orthonormal matrix, the Givens embedding, and the subsampled randomized trigonometric transform (SRTT).
    For the Givens embedding, we apply a sequence of $\lceil 4n\log n\rceil$ uniformly random Givens rotations and subsample to $\ell$ uniformly selected coordinates.
    For the SRTT, we use the discrete cosine transform with rerandomization; see \cite[Def.~B.14]{Epp25a} for details.    
\end{itemize}
All embeddings use the field $\field = \real$, and we report the accuracy parameter 
\begin{equation*}
    \varepsilon = \frac{\expect \norm{\mat{B} - \mat{A}(\mat{\varOmega}^*\mat{A})^+ (\mat{\varOmega}^*\mat{B})}_{\rm F}^2}{\norm{\mat{B} - \mat{A} \mat{A}^+ \mat{B}}_{\rm F}^2} - 1.
\end{equation*}
For each problem instance, each choice of embedding, and each choice of $\ell$, we evaluate this expectation empirically over 1000 trials.

\Cref{fig:universality} shows that all embeddings closely match the predictions of  \cref{thm:partial-isometry} for either the Gaussian or random orthonormal universality class.
Several works are beginning to rigorously explain this universality phenomenon \cite{DL19a,Der23,CDDR24}.

\subsection{Sketch-and-solve: Previous work} \label{sec:sketch-solve-previous}
The standard analysis of sketch-and-solve proceeds in two steps.
First, one establishes with high probability that a random matrix $\mat{\varOmega}$ is a \emph{subspace embedding},
\begin{equation}
\label{eq:subspace_embedding}
    (1 - \eta) \norm{\vec{x}} \leq \norm{\mat{\varOmega}^* \vec{x}} \leq (1 + \eta) \norm{\vec{x}},
    \qquad \text{for all } \vec{x} \in \range([\begin{matrix} \mat{A} & \vec{b} \end{matrix}]).
\end{equation}
Then the subspace embedding property \cref{eq:subspace_embedding} leads to upper bounds for the sketch-and-solve error; see e.g., \cite[Prop.~5.3]{KT24} or \cite[Thm.~C.2]{Epp25a}.
Another version of this analysis \cite{DMMS11,CEMT25} requires a weaker subspace injection property for $\mat{\varOmega}$.
This pattern of analysis applies to many types of embeddings, and it predicts the correct asymptotic error scaling.
However, it is not quantitatively sharp.

The first sharp analysis of sketch-and-solve was provided by Bartan \& Pilanci \cite[Lem.~2.1]{BP20}, who established an exact mean-squared error formula for real Gaussian embeddings; also see the earlier work \cite{AAR21}.
We have restated this result as a special case of \cref{thm:partial-isometry}, and we have extended the result to complex Gaussians as well as real and complex random orthonormal embeddings.

Our approach differs from the statistical, asymptotic analysis of Dobriban \& Liu \cite{DL19a}, who considered a model where $\mat{X}$ is a planted solution and $\mat{B}$ is generated by adding white noise to $\mat{A}\mat{X}$.
Dobriban \& Liu analyzed the sketch-and-solve squared error in the asymptotic limit as $d,\ell,n\to\infty$ with convergent aspect ratios, and they observed differences between Gaussian and random orthonormal universality classes.
In the same vein, Lacotte, Liu, Dobriban, \& Pilanci \cite{LLDP20a} studied the differences between embeddings in the two universality classes as $d,\ell,n\to\infty$ with convergent aspect ratios, and they proved an asymptotic version of the expectation formula in \cref{prop:beta_exp}.
Our work has similar goals to \cite{DL19a,LLDP20a}, yet we emphasize non-asymptotic results leading to precise predictions for any fixed parameters $d,\ell,n$.
Because of this approach, our expectation formulas are tailored to Gaussian and random orthonormal embeddings, and we only offer empirical evidence of universality.

Finally, a couple of works have established sketch-and-solve lower bounds comparable to \cref{thm:sketch-solve-optimality}.
Sridhar, Pilanci, \& \"Ozg\"ur \cite{SPO20} derived a lower bound that applies to any estimator $\Xhat(\mat{\varOmega}^*\mat{A},\mat{\varOmega}^*\mat{B})$ based on real Gaussian embeddings (not just sketch-and-solve).
Later, Bartan \& Pilanci developed additional lower bounds for real Gaussian matrices and several structured embeddings \cite{BP23}.
In contrast, our work does not assume any structure on $\mat{\varOmega}$; we produce a uniform lower bound valid for any data-oblivious embedding.

\section{Results for the randomized SVD} \label{sec:rsvd-description}

The randomized SVD generates a low-rank approximation of any matrix $\mat{A} \in \field^{d \times n}$.
To do so, this algorithm applies a random embedding $\mat{\varOmega} \in \field^{n \times \ell}$ on the right and computes a matrix $\mat{Q}$ whose columns provide an orthonormal basis for $\range(\mat{A} \mat{\varOmega})$.
The algorithm then returns a factored rank-$\ell$ approximation
\begin{equation*}
    \Ahat \coloneqq \mat{Q}(\mat{Q}^*\mat{A}).
\end{equation*}
The approximation can be equivalently written as $\Ahat = \mat{\varPi}_{\mat{A} \mat{\varOmega}} \mat{A}$ where $\mat{\varPi}_{\mat{A} \mat{\varOmega}}$ is the orthogonal projection onto the range of $\mat{A} \mat{\varOmega}$.
Halko, Martinsson, \& Tropp \cite{HMT11} introduced the randomized SVD in its modern form; see \cite[sec.~3]{TW23a} for a discussion of earlier work and related algorithms.

\subsection{Randomized SVD: Theoretical analysis} \label{sec:rsvd-results}

The existing analysis of the randomized SVD compares the error of $\Ahat$ to the error of an optimal rank-$q$ approximation for $q < \ell - \alpha_\field$.
The optimal rank-$q$ approximation of $\mat{A}$ is denoted by $\lowrank{\mat{A}}_q$, and it is generated by any $q$-truncated singular value decomposition.
The result \cite[Thm.~10.5]{HMT11} is as follows.

\begin{theorem}[Randomized SVD: Existing analysis] \label{fact:randomized-svd}
    Let $\mat{\varOmega} \in \field^{n \times \ell}$ be a full-rank random embedding such that $\mat{\varOmega} \sim \mat{U}\mat{\varOmega}$ for each unitary $\mat{U}\in \field^{n \times n}$.
    For any matrix $\mat{A} \in \field^{d \times n}$, 
    \begin{equation*}
        \expect \norm{\mat{A} - \mat{\varPi}_{\mat{A} \mat{\varOmega}} \mat{A}}_{\rm F}^2 \le \min_{q < \ell - \alpha_\field} \left(1 + \frac{q}{\ell - q - \alpha_\field}\right) 
        \norm{\smash{\mat{A} - \lowrank{\mat{A}}_q}}_{\rm F}^2.
    \end{equation*}
\end{theorem}
\noindent In this paper, we obtain a slight refinement. The proof is in \cref{sec:proof_rsvd}.
\begin{theorem}[Randomized SVD: Sharp upper bound] 
\label{thm:rsvd-upper-bound}
    Let $\mat{\varOmega} \in \field^{n \times \ell}$ be a full-rank random embedding such that $\mat{\varOmega} \sim \mat{U}\mat{\varOmega}$ for each unitary $\mat{U}\in \field^{n \times n}$.
    For any $\mat{A} \in \field^{d \times n}$ that satisfies $r \coloneqq \rank(\mat{A}) \geq \ell$,
    \begin{equation*}
        \expect \norm{\mat{A} - \mat{\varPi}_{\mat{A} \mat{\varOmega}} \mat{A}}_{\rm F}^2 \le \min_{q < \ell - \alpha_\field} 
        \frac{r-\ell}{r-q} \left(1 + \frac{q}{\ell - q - \alpha_\field}\right) 
        \norm{\smash{\mat{A} - \lowrank{\mat{A}}_q}}_{\rm F}^2.
    \end{equation*}
\end{theorem}
\noindent \Cref{thm:rsvd-upper-bound} holds for any full-rank rotation-invariant random embedding such as a Gaussian or random orthonormal matrix.
The improvement over the existing analysis is modest, but the bound is tight.
The following matching lower bound is proved in \cref{sec:proof_rsvd_lower}.

\begin{theorem}[Randomized SVD: Sharp lower bound] 
\label{thm:rsvd-lower-bound}
    Let $\mat{\varOmega} \in \field^{n \times \ell}$ be any full-rank random embedding.
    For any problem dimension $d \in \mathbb{N}$ and any rank parameters $q,r \in \mathbb{N}$ with $q + \alpha_{\field} < \ell$ and $\ell \leq r \leq \min\{d,n\}$,
    \begin{equation}
    \label{eq:rsvd-lower-bound}
        \sup_{\substack{\mat{A} \in \field^{d \times n} \\ \rank(\mat{A}) = r}} \frac{\expect \norm{\mat{A} - \mat{\varPi}_{\mat{A} \mat{\varOmega}} \mat{A}}_{\rm F}^2}{\norm{\smash{\mat{A} - \lowrank{\mat{A}}_q}}_{\rm F}^2} \ge \frac{r -\ell}{r - q} \left(1 + \frac{q}{\ell - q - \alpha_\field}\right).
    \end{equation}
\end{theorem}
\noindent 
For each random embedding $\mat{\varOmega}$, \cref{thm:rsvd-lower-bound} shows that there must be a target matrix $\mat{A}$ that drives the error to the level \cref{eq:rsvd-lower-bound} or higher.
Therefore, rotation-invariant embeddings are minimax optimal for the randomized SVD, and \cref{thm:rsvd-upper-bound} is sharp.

\subsection{Randomized SVD: Empirical evidence of universality}\label{sec:rsvd-discussion}

\begin{figure}[t]
    \centering
    \includegraphics[width=0.99\linewidth]{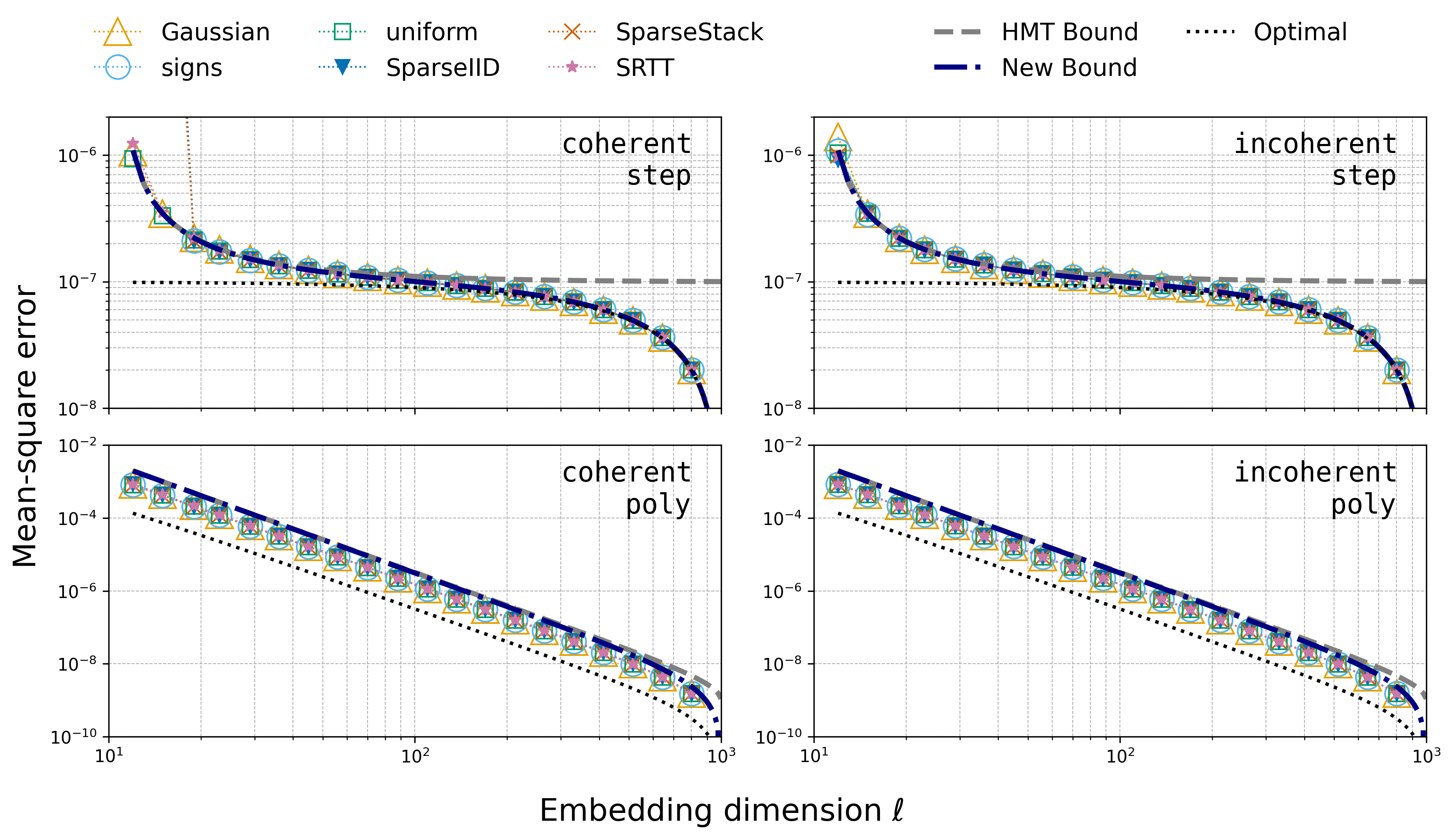}
    \caption{Mean-squared error $\expect \norm{\mat{A} - \Ahat}_{\rm F}^2$ as a function of embedding dimension $\ell$ for six random embeddings.
    The error of the optimal rank-$\ell$ approximation is shown for reference.
    For the ``step'' problems, all embeddings follow the new bound (\cref{thm:rsvd-upper-bound}), except for sign-based embeddings at small values of $\ell$.
    For the ``poly'' problems, the bound predicts the mean-squared error up to a factor of 2.5$\times$.}
    \label{fig:universality-rsvd}
\end{figure}

Many common embeddings with a large embedding dimension can be used interchangeably in the randomized SVD, leading to nearly the same performance.
\Cref{fig:universality-rsvd} demonstrates this universality phenomenon on four problem instances involving a positive-semidefinite (psd) matrix $\mat{A} = \mat{Q}\mat{\varLambda}\mat{Q}^*$.
The target matrix has two possible sets of eigenvectors: $\mat{Q} = \Id_n$ (\emph{left}, ``coherent'') and $\mat{Q}$ equal to the discrete cosine transform (\emph{right}, ``incoherent'').
Additionally, the matrix has two possible sets of eigenvalues:
\begin{align*}
    \mat{\varLambda} &= \twobytwo{\Id_{10}}{\mat{0}}{\mat{0}}{10^{-5}\Id_{n-10}} &&
    \text{(\emph{top}, ``step'')}, \\
    \mat{\varLambda} &= \diag(i^{-2} : i=1,\ldots,n) && \text{(\emph{bottom}, ``poly'')}.
\end{align*}
We set the matrix dimensions $n\coloneqq 10^3$ and test six of the random embeddings from \cref{sec:sketch-solve-discussion}.
We omit the random orthonormal matrix, since it performs identically to a Gaussian embedding, as well as the Givens embedding, since generating this matrix is slow.
We evaluate the expectation $\expect \norm{\mat{A} - \Ahat}_{\rm F}^2$ empirically over $1000$ trials.

The empirical results in \cref{fig:universality-rsvd} confirm that all embeddings with a high embedding dimension $\ell \geq 20$ perform similarly to one another.
However, contrary to the prediction of universality, the sign, SparseStack, and SparseIID embeddings with a small embedding dimension $\ell < 20$ lead to high error in the ``step/coherent'' problem.
These embeddings take discrete $\{-1,0,+1\}$ values, so they can annihilate a vector in the span of the dominant eigenvectors.
Yet, the probability of annihilation vanishes as $\ell$ is increased, and universal behavior is restored in these experiments for $\ell \ge 20$.
Last, in the ``poly'' problems, all the random embeddings perform similarly to one another, presenting more evidence of universality.

The new bound (\cref{thm:rsvd-upper-bound}) is much sharper than the traditional HMT bound (\cref{fact:randomized-svd}) in the ``step'' problems when $\ell \geq n / 10$, and it nearly matches the actual error.
However, in the ``poly'' problems, the new bound overestimates the error of all the embeddings by a factor of roughly 2.5$\times$.
This overestimation factor is not well understood theoretically, as it depends on the detailed distribution of the eigenvalues.

\subsection{Randomized SVD: Previous work}\label{sec:rsvd-previous}

The standard analysis of the randomized SVD is due to Halko, Martinsson, \& Tropp \cite{HMT11}.
Tropp \& Webber \cite{TW23a} developed additional upper bounds.
Yet none of these bounds is as sharp as \cref{thm:rsvd-upper-bound}.

In 2014, Witten \& Cand\`es \cite[Thms.~1.2 \& 3.1]{witten2014randomized} pursued an alternative sharp analysis of the spectral norm error of the randomized SVD.
They considered a full-rank, rotation-invariant embedding $\mat{\varOmega} \in \real^{n\times \ell}$, a dimension $d \in \mathbb{N}$, and a rank parameter $q \in \mathbb{N}$ with $q < \ell - 1$.
Then they constructed a random matrix $\mat{X} \in \real^{(n-q) \times n}$ so the output of the randomized SVD satisfies
\begin{equation*}
    \expect \norm{\smash{\mat{A} - \mat{\varPi}_{\mat{A} \mat{\varOmega}} \mat{A}}} \le \expect \norm{\mat{X}} \, 
    \norm{\smash{\mat{A} - \lowrank{\mat{A}}_q}},
    \qquad \text{for any matrix } \mat{A} \in \real^{d \times n}.
\end{equation*}
Additionally, they proved a matching lower bound.
For any random embedding $\mat{\varOmega} \in \real^{n\times \ell}$,
\begin{equation*}
    \sup_{\substack{\mat{A} \in \field^{d \times n} \\ \mat{A}\neq\lowrank{\mat{A}}_q}} \frac{\expect \norm{\mat{A} - \Ahat}}{\norm{\smash{\mat{A} - \lowrank{\mat{A}}_q}}} \geq \expect \norm{\mat{X}}.
\end{equation*}
Because they worked in the spectral norm, Witten \& Cand\`es gave a construction for the matrix $\mat{X}$, but they were unable to obtain sharp estimates for $\expect \norm{\mat{X}}$.
By contrast, we work in the squared Frobenius norm and obtain sharp error bounds that are rational functions of $r=\rank(\mat{A})$, $\ell$, and $q$.
Our analysis uses a similar technique to Witten \& Cand\`es and considers a worst-case matrix $\mat{A}$ with two distinct eigenvalues, the larger of which is increased to infinity; see \cref{sec:rsvd-nystom-analysis} for details.

\subsection{Connection to Nystr\"om approximation}

The randomized SVD is algebraically connected to the Nystr\"om approximation \cite{WS00,DM05,HMT11,GM13,TYUC17b}, which is a widely used low-rank approximation format for positive-semidefinite (psd) matrices \cite[Lem.~5.1]{TW23a}.

\begin{definition}[Nystr\"om approximation]
    Given an embedding $\mat{\varOmega} \in \field^{n \times \ell}$ and a psd matrix $\mat{H} \in \field^{n\times n}$, the Nystr\"om approximation is 
    \begin{equation*}
        \mat{H} \langle \mat{\varOmega} \rangle = \mat{H}\mat{\varOmega} (\mat{\varOmega}^*\mat{H}\mat{\varOmega})^+ \mat{\varOmega}^*\mat{H}.
    \end{equation*}
\end{definition}

\noindent The Nystr\"om approximation and randomized SVD are related by the \emph{Gram correspondence} \cite{Git11}; see \cite[sec.~2.6]{Epp25a} for history and discussion.

\begin{fact}[Gram correspondence] \label{fact:gram-correspondence}
    Given a random embedding $\mat{\varOmega} \in \field^{n \times \ell}$ and a target matrix $\mat{A} \in \field^{d \times n}$, let $\Ahat \coloneqq (\mat{A} \mat{\varOmega}) (\mat{A} \mat{\varOmega})^+ \mat{A}$ be the randomized SVD approximation of $\mat{A}$ and let $\Hhat \coloneqq \mat{H}\langle \mat{\varOmega} \rangle$ be the Nystr\"om approximation of the Gram matrix $\mat{H} \coloneqq \mat{A}^* \mat{A}$.
    Then
    \begin{equation*}
        \Hhat = \Ahat^* \Ahat \quad \text{and} \quad \tr(\mat{H} - \Hhat) = \norm{\mat{A} - \Ahat}_{\rm F}^2.
    \end{equation*}
\end{fact}

\Cref{fact:gram-correspondence} shows that the randomized Nystr\"om and randomized SVD are algebraically linked, so error bounds extend from one method to the other.
This means that \cref{thm:rsvd-upper-bound,thm:rsvd-lower-bound} are equivalent to the following result for Nystr\"om approximation.

\begin{theorem}[Nystr\"om approximation: Sharp upper and lower bounds] \label{thm:nystrom}
    Let $\mat{\varOmega} \in \field^{n \times \ell}$ be any full-rank random embedding.
    Then $\mat{\varOmega}$ satisfies the following sharp upper and lower bounds.
    \begin{enumerate}[label=(\alph*)]
        \item \label{item:nystrom-result-upper} \textbf{Sharp upper bound.} 
        Assume $\mat{\varOmega} \sim \mat{U} \mat{\varOmega}$ for each unitary $\mat{U}\in \field^{n \times n}$.
        Then for any psd matrix $\mat{H} \in \field^{n \times n}$ with $r = \rank(\mat{H}) \geq \ell$,
        \begin{equation*} 
            \expect \tr(\mat{H} - \mat{H} \langle \mat{\varOmega} \rangle) \le \min_{q < \ell - \alpha_\field} \frac{r-\ell}{r-q} \left(1 + \frac{q}{\ell - q - \alpha_\field}\right) \tr(\mat{H} - \lowrank{\mat{H}}_q).
        \end{equation*}
        \item \label{item:nystrom-result-lower} \textbf{Sharp lower bound.} 
        For any rank parameters $q,r \in \mathbb{N}$ with $q + \alpha_\field < \ell \leq r$ and $r \leq n$,
        \begin{equation*}
            \sup_{\substack{\mat{H} \in \field^{n\times n} \text{ psd} \\ \rank(\mat{H}) = r}} \frac{\expect \tr(\mat{H} - \mat{H}\langle \mat{\varOmega} \rangle)}{\tr(\mat{H} - \lowrank{\mat{H}}_q)} 
            \ge \frac{r-\ell}{r-q}\, \left(1 + \frac{q}{\ell - q - \alpha_\field}\right).
        \end{equation*}
    \end{enumerate}
    
\end{theorem}
\noindent We shall establish the upper and lower bounds for randomized SVD by first proving \cref{thm:nystrom} in \cref{sec:proof_rsvd,sec:proof_rsvd_lower}.

\section{Results for generalized Nystr\"om}
\label{sec:gen_nystrom}

The generalized Nystr{\"o}m approximation \cite{WLRT08,CW09,TYUC17,Nak20} is used in single-pass or streaming applications when the matrix $\mat{A}$ is not stored explicitly. 
It is also used when a cheap, less-accurate low-rank approximation is needed.
The user interacts with $\mat{A}$ only by sketching from the left and right, and these operations can be accelerated with structured random embeddings.
\begin{definition}[Generalized Nystr\"om approximation] \label{def:generalized}
Given random embeddings $\mat{\varOmega} \in \field^{n \times \ell}$ and $\mat{\varPsi} \in \field^{d \times k}$ and a matrix $\mat{A} \in \field^{d \times n}$, the generalized Nystr\"om approximation is 
\begin{equation*}
    \mat{A}\langle \mat{\varOmega},\mat{\varPsi}\rangle \coloneqq \mat{A}\mat{\varOmega}(\mat{\varPsi}^*\mat{A}\mat{\varOmega})^+ \mat{\varPsi}^*\mat{A}.
\end{equation*}
\end{definition}
\noindent The generalized Nystr\"om approximation typically uses uneven embedding dimensions, and we consider the case where the left embedding has more columns than the right embedding, i.e., $k\ge \ell$.
However, by considering the conjugate transpose, we obtain matching error bounds in the opposite case where $\ell \ge k$.

To analyze the generalized Nystr\"om approximation, we use the following observation \cite{CW09,TYUC17,CEMT25}.
The generalized Nystr\"om approximation can be realized by first generating a rank-$\ell$ randomized SVD approximation
\begin{equation*} 
    \Ahat \coloneqq (\mat{A}\mat{\varOmega})(\mat{A}\mat{\varOmega})^+ \mat{A},
\end{equation*}
and then approximating the action of the pseudoinverse with the sketch-and-solve method,
\begin{equation*}
    (\mat{A}\mat{\varOmega})^+ \mat{A} \approx (\mat{\varPsi}^*\mat{A}\mat{\varOmega})^+ \mat{\varPsi}^*\mat{A}.%
\end{equation*}
Therefore, the generalized Nystr\"om approximation combines the randomized SVD and sketch-and-solve approaches analyzed in the previous sections.

\subsection{Sharp error bounds for generalized Nystr\"om}

Tropp, Yurtsever, Udell, and Cevher \cite[Thm.~4.3]{TYUC17} derived explicit upper bounds on the error of the generalized Nystr\"om approximation for Gaussian embeddings; see also \cite[Thm.~4.7]{CW09}.
Here, we improve on these past works by obtaining sharp upper and lower bounds given the rank of $\mat{A}$ and dimensions of $\mat{\varOmega}$ and $\mat{\varPsi}$.

\begin{theorem}[Generalized Nystr\"om: Sharp upper and lower bounds] \label{thm:gen_nystrom}
    Let $\mat{\varOmega} \in \field^{n \times \ell}$ and $\mat{\varPsi}\in\field^{d \times k}$ be full-rank random embeddings with $\ell + \alpha_{\field} < k$.
    Then $\mat{\varOmega}$ and $\mat{\varPsi}$ satisfy the following error bounds.
    \begin{enumerate}[label=(\alph*)]
        \item \textbf{Sharp upper bound.}
        Assume $\mat{\varOmega} \sim \mat{U}\mat{\varOmega}$ for each unitary $\mat{U}\in \field^{n \times n}$ and that $\mat{\varPsi}$ is independent of $\mat{\varOmega}$ and is either Gaussian or random orthonormal.
        Then for any matrix $\mat{A} \in \field^{d \times n}$ with $r = \rank(\mat{A}) \geq k$,
        \begin{equation*}
            \expect \norm{\mat{A} - \mat{A}\langle\mat{\varOmega},\mat{\varPsi}\rangle}_{\rm F}^2 
            \le \left(1 + \gamma \, \frac{\ell}{k - \ell - \alpha_\field} \right) 
            \min_{q < \ell - \alpha_\field} \left[\frac{r-\ell}{r-q} \left( 1 + \frac{q}{\ell-q-\alpha_\field} \right) 
            \norm{\smash{\mat{A} - \lowrank{\mat{A}}_q}}_{\rm F}^2\right].
        \end{equation*}
        Here $\gamma = 1$ if $\mat{\varPsi}$ is a Gaussian embedding and $\gamma = (d - k) / (d - \ell)$ if $\mat{\varPsi}$ is a random orthonormal embedding.
        \item \label{item:gen-nystrom-result-lower} \textbf{Sharp lower bound.} For any rank parameters $q, r \in \mathbb{N}$ with $q + \alpha_\field < \ell$ and $k \leq r \leq \min\{d,n\}$,
        \begin{equation*}
            \sup_{\substack{\mat{A} \in \field^{d\times n} \\ \rank(\mat{A}) = r}} 
            \frac{\expect \norm{\mat{A} - \mat{A}\langle\mat{\varOmega},\mat{\varPsi}\rangle}_{\rm F}^2}
            {\norm{\smash{\mat{A} - \lowrank{\mat{A}}_q}}_{\rm F}^2} 
            \ge \left(1 + \frac{d-k}{d-\ell} \frac{\ell}{k - \ell - \alpha_\field} \right)
            \frac{r-\ell}{r-q} \left( 1 + \frac{q}{\ell-q-\alpha_\field} \right).
        \end{equation*}
    \end{enumerate}
\end{theorem}
\begin{proof}
We can prove the generalized Nystr\"om approximation upper bound by combining error bounds for sketch-and-solve (\cref{thm:partial-isometry}) and the randomized SVD (\cref{thm:rsvd-upper-bound}).
\begin{align*}
    \expect \norm{\mat{A} - \mat{A}\langle\mat{\varOmega},\mat{\varPsi}\rangle}_{\rm F}^2 
    &= \left(1 + \gamma \,  \frac{\ell}{k - \ell - \alpha_\field} \right)\expect \norm{\smash{\mat{A} - \mat{A}\mat{\varOmega}(\mat{A}\mat{\varOmega})^+ \mat{A}}}_{\rm F}^2 \\ 
    &\le \left(1 + \gamma\, \frac{\ell}{k - \ell - \alpha_\field} \right) 
    \min_{q < \ell-\alpha_\field} \frac{r-\ell}{r-q} \left( 1 + \frac{q}{\ell-q-\alpha_\field} \right) 
    \norm{\smash{\mat{A} - \lowrank{\mat{A}}_q}}_{\rm F}^2.
\end{align*}
This completes the proof of the upper bound. The proof of the lower bound is more challenging, and it is given in \cref{sec:proof_gen_nystrom_lower}.
\end{proof}

\subsection{Implications}
We discuss two implications of this result.
First, we explore how these results suggest we should pick the two embedding dimensions $k$ and $\ell$.
Second, we compare the error bounds for generalized Nystr\"om and the randomized SVD.

\paragraph{How to pick $k$ and $\ell$?}
To discuss the error bounds in \cref{thm:gen_nystrom}, we consider the complex case and take the supremum over all matrices $\mat{A}$ of any dimension:
\begin{equation}
\label{eq:minimax}
    \sup_{
    \substack{\mat{A} \in \complex^{d \times n} \, d \geq k,\, n \geq \ell\\ \rank(\mat{A}) > q}}
    \frac{\expect \norm{\mat{A} - \mat{A}\langle\mat{\varOmega}_n,\mat{\varPsi}_d\rangle}_{\rm F}^2}{\norm{\smash{\mat{A} - \lowrank{\mat{A}}_q}}_{\rm F}^2}
    \geq \biggl(1 + \frac{\ell}{k-\ell}\biggr)\biggl(1 + \frac{q}{\ell-q}\biggr).
\end{equation}
Here, $\mat{\varOmega}_n$ and $\mat{\varPsi}_d$ denote some family of $n\times \ell$ and $d\times k$ random embeddings.
The bound \cref{eq:minimax} isolates three essential quantities: the rank $k$ of the left embedding, the rank $\ell$ of the right embedding, and the target rank $q$ for the generalized Nystr{\'o}m approximation.
This bound provides insight into the best choice of oversampling ratio $k/\ell$ between the two sketching dimensions.
There are small disagreements about this matter in the literature, with Tropp et al.~\cite[pg.~1456]{TYUC17} suggesting $k/\ell = 2$ and Nakatsukasa \cite[Alg.~2.1]{Nak20} advocating a smaller ratio $k/\ell = 1.5$.
At meetings, one hears whispers of even smaller oversampling ratios, approaching the minimal value $k/\ell = 1$.

Here is one way to formalize the selection of $k$ versus $\ell$. Given a target rank $q$ and a fixed budget of matrix--vector multiplications $t = k + \ell$, how should one pick $k$ and $\ell$ to guarantee
\begin{equation*}
    \expect \norm{\mat{A} - \mat{A}\langle\mat{\varOmega},\mat{\varPsi}\rangle}_{\rm F}^2
    \leq (1 + \varepsilon) \norm{\smash{\mat{A} - \lowrank{\mat{A}}_q}}_{\rm F}^2
\end{equation*}
with the smallest possible value of $\varepsilon$?
Treating $k$ and $\ell$ as continuous parameters and minimizing the right-hand side of \cref{eq:minimax}, we arrive at the answer
\begin{equation}
\label{eq:ratio}
    \frac{k}{\ell} = \frac{\sqrt{t - q}}{\sqrt{q}}.
\end{equation}
We obtain the optimal integer values by either rounding $k$ up and $\ell$ down or the reverse.

From this exercise, we see that different choices of the oversampling parameter are justified in different settings.
When the singular values decrease rapidly, we want the target rank $q$ as large as possible, so we should use an aggressive oversampling ratio $k/\ell \approx 1$.
When the singular values decrease more slowly, we want a smaller $q$ with a better accuracy parameter $\varepsilon$, leading to a larger optimal value of $k/\ell$.
These predictions, justified by appealing to worst-case error bounds, appear consistent with past advice \cite{TYUC17}.

\paragraph{Generalized Nystr\"om vs.\ randomized SVD}

The lower bound \cref{eq:minimax} shows that the generalized Nystr\"om approximation needs a large number of matrix--vector multiplications to guarantee
\begin{equation*}
    \expect \norm{\mat{A} - \mat{A}\langle\mat{\varOmega},\mat{\varPsi}\rangle}_{\rm F}^2
    \leq (1 + \varepsilon) \norm{\smash{\mat{A} - \lowrank{\mat{A}}_q}}_{\rm F}^2
\end{equation*}
for small $\varepsilon > 0$.
Indeed, the necessary number of multiplications is at least 
\begin{equation*}
    k + \ell \geq 2 
    \left(\frac{1}{\varepsilon} + 1 \right) 
    \left(\frac{1}{\varepsilon} + \sqrt{\frac{1}{\varepsilon^2} + \frac{2}{\varepsilon}} + 1\right) q 
    \ge \left(\frac{4}{\varepsilon^2} + 2 \right)q.
\end{equation*}
We obtain this expression by plugging in the optimal oversampling ratio \cref{eq:ratio} and then inverting the lower bound in \cref{eq:minimax}.
By comparison, the randomized SVD achieves the guarantee
\begin{equation*}
    \expect \norm{\mat{A} - \mat{\varPi}_{\mat{A} \mat{\varOmega}} \mat{A}}_{\rm F}^2
    \leq (1 + \varepsilon) 
    \norm{\smash{\mat{A} - \lowrank{\mat{A}}_q}}_{\rm F}^2.
\end{equation*}
as soon as the number of matrix--vector multiplications is 
\begin{equation*}
2 \ell \geq \left(\frac{2}{\varepsilon} + 2\right)q.
\end{equation*}
Thus the number of matrix--vector multiplications for the randomized SVD scales like $q/\varepsilon$ while the requirements for generalized Nystr\"om scale like $q/\varepsilon^2$.

\section{Analysis of sketch-and-solve}
\label{sec:sketch_analysis}

This section analyzes sketch-and-solve.
\Cref{sec:sketch-solve-formula} establishes a deterministic error decomposition for sketch-and-solve.
Then \cref{sec:wishart,sec:beta} establish expectation formulas for the Wishart and Beta random matrices.
Last, \cref{sec:partial-isometry-proof} establishes expectation formulas for sketch-and-solve and
\cref{sec:proof_optimality} establishes our lower bound for the sketch-and-solve error.
We repeatedly use rotation invariance in our arguments.

\subsection{Error decomposition} \label{sec:sketch-solve-formula}

We begin with a linear algebraic error decomposition for sketch-and-solve.
Versions of this decomposition have appeared in the literature \cite{TYUC17,CEMT25}, but we are unaware of a reference that handles the general case where the matrix $\mat{A}$ could be rank-deficient.

\begin{lemma}[Sketch-and-solve: Error decomposition] \label{lem:sketch-solve-formula}
    Let $\mat{\varOmega} \in \field^{n\times \ell}$ be an embedding, let $\mat{A} \in \field^{n\times d}$ and $\mat{B} \in \field^{n\times p}$ be matrices, and assume 
    \begin{equation*}
    r \coloneqq \rank(\mat{A}) = \rank(\mat{\varOmega}^* \mat{A}).
    \end{equation*}
    Let $\mat{A} = \mat{Q}\mat{R}$ be any decomposition which factors $\mat{A}$ as a product of a matrix $\mat{Q} \in \field^{n\times r}$ with orthonormal columns and a matrix $\mat{R} \in \field^{r\times d}$ with linearly independent rows, and extend $\mat{Q}$ to a square unitary matrix $[\begin{matrix} \mat{Q} & \mat{Q}_\perp \end{matrix}] \in \field^{n\times n}$.
    Then the difference between the sketch-and-solve solution $\Xhat = (\mat{\varOmega}^*\mat{A})^+(\mat{\varOmega}^*\mat{B})$ and the exact solution $\mat{X}_\star = \mat{A}^+\mat{B}$ is
    \begin{equation} \label{eq:sketch-and-solve-error-decomposition-1}
        \Xhat - \mat{X}_\star = \mat{R}^+ (\mat{\varOmega}^*\mat{Q})^+(\mat{\varOmega}^*\mat{Q}_\perp^{\vphantom{*}}) \mat{Q}_\perp^*\mat{B}
    \end{equation}
    and the residual norm is
    \begin{equation} \label{eq:sketch-and-solve-error-decomposition-2}
        \norm{\mat{B} - \mat{A}\Xhat}_{\rm F}^2 = \norm{\smash{(\mat{\varOmega}^*\mat{Q})^+(\mat{\varOmega}^*\mat{Q}_\perp^{\vphantom{*}}) \mat{Q}_\perp^*\mat{B}}}_{\rm F}^2 + 
        \norm{\mat{Q}_\perp^*\mat{B}}_{\rm F}^2.
    \end{equation}
    In these formulas, the optimal least-squares error is $\norm{\smash{\mat{Q}_\perp^*\mat{B}}}_{\rm F}^2 = \norm{\smash{\mat{B} - \mat{A} \mat{X}_{\star}}}_{\rm F}^2$.
\end{lemma}

\begin{proof}
By assumption, $\rank(\mat{A}) = \rank(\mat{\varOmega}^* \mat{A})$ and thus $\range(\mat{A}^* \mat{\varOmega}) = \range(\mat{A}^*)$.
Consequently,
\begin{equation*}
    (\mat{\varOmega}^* \mat{A})^+ (\mat{\varOmega}^* \mat{A})
    = \mat{\varPi}_{\mat{A}^* \mat{\varOmega}}
    = \mat{\varPi}_{\mat{A}^*}.
\end{equation*}
Now, apply the decomposition of identity $\Id = \mat{Q}\mat{Q}^* + \mat{Q}_\perp^{\vphantom{*}} \mat{Q}_\perp^*$ to compute
\begin{align*}
    \Xhat &= (\mat{\varOmega}^*\mat{A})^+ \mat{\varOmega}^*\mat{B} \\
    &= (\mat{\varOmega}^*\mat{A})^+ \mat{\varOmega}^*\bigl[\mat{Q} \mat{Q}^* + \mat{Q}_{\perp}^{\vphantom{*}} \mat{Q}_{\perp}^*\bigr] \mat{B} \\
    &= (\mat{\varOmega}^*\mat{A})^+ \mat{\varOmega}^*\bigl[\mat{A} \mat{A}^+ + \mat{Q}_{\perp}^{\vphantom{*}} \mat{Q}_{\perp}^*\bigr] \mat{B} \\
    &= \mat{\varPi}_{\mat{A}^*} \mat{A}^+ \mat{B}
    + (\mat{\varOmega}^*\mat{A})^+ (\mat{\varOmega}^* \mat{Q}_{\perp}^{\vphantom{*}}) \mat{Q}_{\perp}^* \mat{B} \\
    &= \mat{A}^+ \mat{B} + (\mat{\varOmega}^*\mat{A})^+ (\mat{\varOmega}^* \mat{Q}_{\perp}^{\vphantom{*}}) \mat{Q}_{\perp}^* \mat{B}.
\end{align*}
Here we have used the fact that $\mat{Q} \mat{Q}^* = \mat{A} \mat{A}^+$ and $\mat{\varPi}_{\mat{A}^*} \mat{A}^+ = \mat{A}^+$.
Next, we make one further simplification.
\begin{equation*}
    (\mat{\varOmega}^*\mat{A})^+
    = (\mat{\varOmega}^* \mat{Q} \mat{R})^+
    = \mat{R}^+ (\mat{\varOmega}^* \mat{Q})^+.
\end{equation*}
This identity holds because $\mat{\varOmega}^* \mat{Q} \in \field^{\ell \times r}$ has linearly independent columns and $\mat{R} \in \field^{r \times d}$ has linearly independent rows.
We have established the error formula \cref{eq:sketch-and-solve-error-decomposition-1}.

Since the least-squares residual $\mat{B} - \mat{A}\mat{X}_\star$ is orthogonal to the range of $\mat{A}$,
\begin{equation*}
    \norm{\smash{\mat{B} - \mat{A}\Xhat}}_{\rm F}^2 = \norm{\mat{B} - \mat{A}\mat{X}_\star}_{\rm F}^2 + \norm{\smash{\mat{A}(\Xhat - \mat{X}_\star)}}_{\rm F}^2.
\end{equation*}
Using the decomposition of identity $\Id = \mat{Q}\mat{Q}^* + \mat{Q}_\perp^{\vphantom{*}} \mat{Q}_\perp^*$, we compute the first term.
\begin{equation*}
    \norm{\mat{B} - \mat{A}\mat{X}_\star}_{\rm F}
    = \norm{\smash{(\Id - \mat{A}\mat{A}^+) \mat{B}}}_{\rm F} 
    = \norm{\smash{(\Id - \mat{Q}\mat{Q}^*) \mat{B}}}_{\rm F} 
    = \norm{\smash{\mat{Q}_\perp^{\vphantom{*}}\mat{Q}_\perp^* \mat{B}}}_{\rm F} 
    = \norm{\smash{\mat{Q}_\perp^* \mat{B}}}_{\rm F}.
\end{equation*}
Next we compute the second term.
\begin{equation*}
    \norm{\smash{\mat{A}(\Xhat - \mat{X}_\star)}}_{\rm F} = \norm{\smash{\mat{Q}\mat{R}\mat{R}^+}(\mat{\varOmega}^*\mat{Q})^+(\mat{\varOmega}^*\mat{Q}_\perp^{\vphantom{*}}) \mat{Q}_\perp^*\mat{B}}_{\rm F} = \norm{\smash{(\mat{\varOmega}^*\mat{Q})^+(\mat{\varOmega}^*\mat{Q}_\perp^{\vphantom{*}}) \mat{Q}_\perp^*\mat{B}}}_{\rm F}.
\end{equation*}
The first identity is the decomposition $\mat{A} = \mat{Q}\mat{R}$ and the error formula \cref{eq:sketch-and-solve-error-decomposition-1}.
The second identity uses rotational invariance of the Frobenius norm and the identity $\mat{R}\mat{R}^+ = \Id$ for the matrix $\mat{R}$ with linearly independent rows.
The conclusion \cref{eq:sketch-and-solve-error-decomposition-2} follows.
\end{proof}

In the sections to follow, our basic strategy for analyzing sketch-and-solve is to substitute different values for the random embedding $\mat{\varOmega}$.
Then we perform linear algebraic manipulations to determine the expectation of $\norm{\smash{(\mat{\varOmega}^*\mat{Q})^+(\mat{\varOmega}^*\mat{Q}_\perp^{\vphantom{*}}) \mat{Q}_\perp^*\mat{B}}}_{\rm F}^2$.
To that end, we need a couple facts about Wishart and Beta random matrices that will be described in the next two sections.

\subsection{Wishart random matrices} \label{sec:wishart}

The first type of random matrix that will be used in the analysis is a Wishart matrix.

\begin{definition}[Wishart matrix]
    A Wishart matrix with parameters $(r, \ell)$ is a matrix
    \begin{equation*}
    \mat{W} = \mat{G}\mat{G}^*,
    \end{equation*}
    where $\mat{G} \in \field^{r \times \ell}$ is populated with iid $\mathcal{N}_\field(0,1)$ random variables.
    We write $\mat{W} \sim \Wishart_\field(r,\ell)$.
\end{definition}

\noindent The expected value of an inverse Wishart matrix is given in \cite{Sha80} and \cite[p.~97]{Mui82}.

\begin{fact}[Expected value of inverse Wishart matrix] \label{fact:wishart}
    If $\mat{W} \sim \Wishart_\field(r,\ell)$ with $r + \alpha_{\field} < \ell$, then
    \begin{equation*}
    \expect\mat{W}^{-1} = \frac{1}{\ell - r - \alpha_\field} \,\Id.
    \end{equation*}
\end{fact}

\noindent \Cref{fact:wishart} can be used to derive a sketching identity for Gaussian embeddings; see \cite[Prop.~10.1 \& 10.2]{HMT11} or \cite[Prop.~8.6]{TW23a}.

\begin{proposition}[Gaussian sketching identity] \label{prop:gaussian-sketch}
    Consider a Gaussian embedding \begin{equation*}
    \mat{\varOmega} = \twobyone{\mat{\varOmega}_1}{\mat{\varOmega}_2} \in \field^{n \times \ell}
    \end{equation*}
    that is partitioned into the first $r$ and last $n-r$ rows, where $r + \alpha_{\field} < \ell$.
    For any matrix $\mat{B} \in \field^{(n-r) \times p}$,
    \begin{equation*}
        \expect\, \norm{\smash{(\mat{\varOmega}_1^*)^+ \mat{\varOmega}_2^* \mat{B} }}_{\rm F}^2 = \frac{r}{\ell - r - \alpha_\field} \, \norm{\mat{B}}_{\rm F}^2.
    \end{equation*}
\end{proposition}

\subsection{Beta random matrices} \label{sec:beta}

The second type of random matrix that will be used in the analysis is the Beta random matrix of type I \cite{Mit70}, also known as the Jacobi random matrix.

\begin{definition}[Beta random matrix] \label{def:beta}
    A Beta random matrix with parameters $(r,\ell,n)$ where $r \leq n$ takes the form
    \begin{equation} \label{eq:wishart-decomposition}
        \mat{X} = (\mat{W}_1+\mat{W}_2)^{-1/2}\mat{W}_1(\mat{W}_1+\mat{W}_2)^{-1/2},
    \end{equation}
    where $\mat{W}_1 \sim \Wishart_\field(r,\ell)$ and $\mat{W}_2\sim \Wishart_\field(r,n-\ell)$ are independent random variables.
    We write $\mat{X} \sim \Beta_\field(r,\ell,n)$.
\end{definition}

\noindent We can directly compute the expected value of an inverse Beta random matrix.

\begin{proposition}[Expected value of inverse Beta matrix] \label{prop:beta_exp}
    If $\mat{X} \sim \Beta_\field(r,\ell,n)$ with $r + \alpha_{\field} < \ell$, then
    \begin{equation*}
        \expect \mat{X}^{-1} = \left(1 + \frac{n-\ell}{\ell - r - \alpha_\field}\right) \Id.
    \end{equation*}
\end{proposition}

\begin{proof}
    The matrix $\mat{X} \in \field^{r \times r}$ satisfies
    \begin{equation*}
        \mat{U}^*\mat{X}\mat{U} \sim \mat{X} \quad \text{if } \mat{U} \in \field^{r \times r} \text{ is unitary}.
    \end{equation*}
    It follows that $\expect\mat{X}^{-1} = [\expect \tr(\mat{X}^{-1}) / r]\,\Id$.
    Using \cref{def:beta} for the Beta random matrix model, we compute
    \begin{equation*}
        \tr(\mat{X}^{-1}) = \tr[(\mat{W}_1+\mat{W}_2)^{1/2}\mat{W}_1^{-1}(\mat{W}_1+\mat{W}_2)^{1/2}] = \tr[\mat{W}_1^{-1}(\mat{W}_1+\mat{W}_2)] = r + \tr(\mat{W}_1^{-1}\mat{W}_2^{\vphantom{-1}}).
    \end{equation*}
    Using the independence of $\mat{W}_1$ and $\mat{W}_2$, the expected value of the inverse Wishart matrix (\cref{fact:wishart}), and the fact that $\mat{W}_2$ is the sum of $n-\ell$ outer products of Gaussian vectors,
    \begin{equation*}
        \expect \tr(\mat{X}^{-1}) = r + \tr\bigl( \expect\mat{W}_1^{-1}\,\expect\mat{W}_2^{\vphantom{-1}} \bigr) 
        = r + \tr\left( \frac{n-\ell}{\ell - r - \alpha_\field} \Id \right) = \left(1 + \frac{n-\ell}{\ell - r - \alpha_\field}\right) r.
    \end{equation*}
    This completes the proof.
\end{proof}

We now connect the Beta random matrix to a random orthonormal embedding $\mat{\varOmega}$.
First, partition $\mat{\varOmega}$ into its first $r$ and last $n - r$ rows:
\begin{equation*}
\mat{\varOmega} = \begin{bmatrix} \mat{\varOmega}_1 \\ \mat{\varOmega}_2 \end{bmatrix} \in \field^{n \times \ell}.
\end{equation*}
We observe that $\mat{\varOmega}_1$ can be realized as the top left submatrix of a random unitary matrix
\begin{equation*}
    \mat{U} = \twobytwo{\mat{\varOmega}_1}{\mat{\varOmega}_3}{\mat{\varOmega}_2}{\mat{\varOmega}_4} \in \field^{n\times n}.
\end{equation*}
On the basis of this observation, $\mat{\varOmega}_1$ can be generated by drawing matrices $\mat{G}_1 \in \field^{r \times \ell}$ and $\mat{G}_3 \in \field^{r \times (n - \ell)}$ with iid $\mathcal{N}_\field(0,1)$ entries, and applying the polar decomposition
\begin{equation*}
    \onebytwo{\mat{G}_1}{\mat{G}_3} = \mat{W}^{1/2} \onebytwo{\mat{\varOmega}_1}{\mat{\varOmega}_3} \quad \text{where} \quad \mat{W} = \mat{G}_1 \mat{G}_1^* + \mat{G}_3 \mat{G}_3^*.
\end{equation*}
If we define $\mat{W}_1^{\vphantom{*}} \coloneqq \mat{G}_1 \mat{G}_1^*$ and $\mat{W}_2^{\vphantom{*}} \coloneqq \mat{G}_3 \mat{G}_3^*$, then it follows
\begin{equation*}
    \mat{\varOmega}_1 \mat{\varOmega}_1^* 
    = \mat{W}^{-1/2}\mat{G}_1 \mat{G}_1^* \mat{W}^{-1/2} = (\mat{W}_1+\mat{W}_2)^{-1/2}\mat{W}_1(\mat{W}_1+\mat{W}_2)^{-1/2},
\end{equation*}
where $\mat{W}_1 \sim \Wishart_\field(r,\ell)$ and $\mat{W}_2\sim \Wishart_\field(r,n-\ell)$ are independent.
Therefore we have shown $\mat{\varOmega}_1 \mat{\varOmega}_1^* \sim \Beta_\field(r,\ell,n)$.
Using this connection, we now establish a useful sketching identity for random orthonormal matrices.

\begin{proposition}[Random orthonormal sketching identity] \label{prop:partial-isom}
    Consider a random orthonormal embedding 
    \begin{equation*}
    \mat{\varOmega} = \twobyone{\mat{\varOmega}_1}{\mat{\varOmega}_2} \in \field^{n \times \ell}
    \end{equation*}
    that is partitioned into its first $r$ and last $n-r$ rows, where $r + \alpha_{\field} < \ell$.
    For any invertible matrix $\mat{R} \in \field^{\ell \times \ell}$ and any matrix $\mat{B} \in \field^{(n-r) \times p}$, 
    \begin{equation*}
        \expect \norm{\smash{ (\mat{R}^* \mat{\varOmega}_1^*)^+ \mat{R}^* \mat{\varOmega}_2^* \mat{B}}}_{\rm F}^2 \geq \frac{n-\ell}{n-r} \, \frac{r}{\ell - r - \alpha_\field} \, \norm{\mat{B}}_{\rm F}^2.
    \end{equation*}
    This lower bound holds with equality when $\mat{R} = \Id$.
\end{proposition}
\begin{proof}
Since the expectation and the trace commute,
\begin{align*}
    \expect \norm{(\mat{R}^* \mat{\varOmega}_1^*)^+ \mat{R}^* \mat{\varOmega}_2^* \mat{B}}_{\rm F}^2
    &= \expect \tr\bigl(\mat{B}^* \mat{\varOmega}_2^{\vphantom{*}} \mat{R} (\mat{\varOmega}_1^{\vphantom{*}} \mat{R})^{+} (\mat{R}^* \mat{\varOmega}_1^*)^{+} \mat{R}^* \mat{\varOmega}_2^* \mat{B}\bigr) \\
    &= \tr\bigl(\mat{B}^* \expect\bigl[ \mat{\varOmega}_2^{\vphantom{*}} \mat{R} (\mat{\varOmega}_1^{\vphantom{*}} \mat{R})^{+} (\mat{R}^* \mat{\varOmega}_1^*)^{+} \mat{R}^* \mat{\varOmega}_2^* \bigr] \mat{B}\bigr).
\end{align*}
The matrix $\mat{X} \coloneqq \mat{\varOmega}_2^{\vphantom{*}} \mat{R} (\mat{\varOmega}_1^{\vphantom{*}} \mat{R})^{+} (\mat{R}^* \mat{\varOmega}_1^*)^{+} \mat{R}^* \mat{\varOmega}_2^*$ satisfies
\begin{equation*}
\mat{U}^*\mat{X}\mat{U} \sim \mat{X} \quad \text{if } \mat{U} \in \field^{(n-r) \times (n-r)} \text{ is unitary}.
\end{equation*}
It follows that $\expect \mat{X} = [\expect \tr (\mat{X}) / (n-r) ] \,\Id$ and therefore
\begin{equation*}
    \tr\bigl(\mat{B}^* \expect\bigl[ \mat{\varOmega}_2^{\vphantom{*}} \mat{R} (\mat{\varOmega}_1^{\vphantom{*}} \mat{R})^{+} (\mat{R}^* \mat{\varOmega}_1^*)^{+} \mat{R}^* \mat{\varOmega}_2^* \bigr] \mat{B}\bigr)
    = \frac{ \expect \tr\bigl( \mat{\varOmega}_2^{\vphantom{*}} \mat{R} (\mat{\varOmega}_1^{\vphantom{*}} \mat{R})^{+} (\mat{R}^* \mat{\varOmega}_1^*)^{+} \mat{R}^* \mat{\varOmega}_2^* \bigr)}{n-r}
    \tr\bigl(\mat{B}^* \mat{B} \bigr).
\end{equation*}
We know that $\tr\bigl(\mat{B}^* \mat{B} \bigr) = \norm{\mat{B}}_{\rm F}^2$.
Further, since $\mat{\varOmega}$ has orthonormal columns, we know that $\Id = \mat{\varOmega}^* \mat{\varOmega} = \mat{\varOmega}_1^* \mat{\varOmega}_1^{\vphantom{*}} + \mat{\varOmega}_2^* \mat{\varOmega}_2^{\vphantom{*}}$ and consequently
\begin{align*}
    \tr\bigl( (\mat{\varOmega}_1^{\vphantom{*}} \mat{R})^{+} (\mat{R}^* \mat{\varOmega}_1^*)^{+} \mat{R}^* \mat{\varOmega}_2^* \mat{\varOmega}_2^{\vphantom{*}} \mat{R} \bigr) 
    &= \tr\bigl( (\mat{\varOmega}_1^{\vphantom{*}} \mat{R})^{+} (\mat{R}^* \mat{\varOmega}_1^*)^{+} \mat{R}^* \mat{R} \bigr)
    - \tr\bigl( (\mat{\varOmega}_1^{\vphantom{*}} \mat{R})^{+} (\mat{R}^* \mat{\varOmega}_1^*)^{+} \mat{R}^* \mat{\varOmega}_1^* \mat{\varOmega}_1^{\vphantom{*}} \mat{R} \bigr) \\
    &= \norm{ (\mat{R}^* \mat{\varOmega}_1^*)^{+} \mat{R}^* }_{\rm F}^2
    - \tr\bigl( \mat{\varOmega}_1^{\vphantom{*}} \mat{R} (\mat{\varOmega}_1^{\vphantom{*}} \mat{R})^{+} (\mat{R}^* \mat{\varOmega}_1^*)^{+} \mat{R}^* \mat{\varOmega}_1^* \bigr).
\end{align*}
The matrix $\mat{\varOmega}_1^{\vphantom{*}} \mat{R} (\mat{\varOmega}_1^{\vphantom{*}} \mat{R})^{+} (\mat{R}^* \mat{\varOmega}_1^*)^{+} \mat{R}^* \mat{\varOmega}_1^*$ is an orthogonal projection and thus
\begin{equation*}
    \tr\bigl( \mat{\varOmega}_1^{\vphantom{*}} \mat{R} (\mat{\varOmega}_1^{\vphantom{*}} \mat{R})^{+} (\mat{R}^* \mat{\varOmega}_1^*)^{+} \mat{R}^* \mat{\varOmega}_1^* \bigr)
    = \tr\bigl(\mat{\varPi}_{\mat{\varOmega}_1 \mat{R}} \bigr) 
    = \rank(\mat{\varOmega}_1 \mat{R})
    = r.
\end{equation*}
In summary, we have established the equality
\begin{equation*}
    \expect \norm{\smash{ (\mat{R}^* \mat{\varOmega}_1^*)^+ \mat{R}^* \mat{\varOmega}_2^* \mat{B}}}_{\rm F}^2 = \frac{\expect \norm{ (\mat{R}^* \mat{\varOmega}_1^*)^{+} \mat{R}^* }_{\rm F}^2 - r}{n-r} \, \norm{\mat{B}}_{\rm F}^2.
\end{equation*}
To finish the proof, we apply the ``weighting hurts'' lemma (\cref{lem:weighting_hurts}) that appears at the end of this section
with
$\mat{M}=\mat{\varOmega}_1^*\in\field^{\ell\times r}$. For any matrix $\mat{\varOmega}_1 \in \field^{r \times \ell}$ and any invertible matrix $\mat{R} \in \field^{\ell \times \ell}$,
\begin{equation*}
    \norm{ (\mat{R}^* \mat{\varOmega}_1^*)^{+} \mat{R}^* }_{\rm F}^2
    \geq \norm{ (\mat{\varOmega}_1^*)^{+}}_{\rm F}^2
\end{equation*}
Since $\mat{\varOmega}_1$ has linearly independent rows, we can simplify further.
\begin{equation*}
    \norm{ (\mat{\varOmega}_1^*)^{+} }_{\rm F}^2
    = \tr\bigl( (\mat{\varOmega}_1^*)^{+} \mat{\varOmega}_1^{+} \bigr)
    = \tr \bigl((\mat{\varOmega}_1^{\vphantom{*}} \mat{\varOmega}_1^*)^{-1}\bigr).
\end{equation*}
Since $\mat{\varOmega}_1 \mat{\varOmega}_1^* \sim \Beta_\field(r,\ell,n)$, we use \cref{prop:beta_exp} to calculate
\begin{equation*}
    \expect \tr \bigl((\mat{\varOmega}_1^{\vphantom{*}} \mat{\varOmega}_1^*)^{-1}\bigr)
    = \tr \bigl( \expect (\mat{\varOmega}_1^{\vphantom{*}} \mat{\varOmega}_1^*)^{-1}\bigr)
    = \tr \Biggl(\Biggl[1 + \frac{n -\ell}{\ell - r - \alpha_{\field}}\Biggr] \Id \Biggr)
    = \Biggl[1 + \frac{n -\ell}{\ell - r - \alpha_{\field}}\Biggr] r.
\end{equation*}
This establishes the lower bound error formula and completes the proof.
\end{proof}

We conclude by proving the ``weighting hurts'' lemma used in the proof of \cref{prop:partial-isom}.

\begin{lemma}[Weighting hurts] \label{lem:weighting_hurts}
Consider a matrix $\mat{M} \in \field^{\ell \times r}$ and an invertible matrix $\mat{R} \in \field^{\ell \times \ell}$.
Then
\begin{equation}
\label{eq:to_verify}
    \norm{\smash{\mat{M}^+}}^2_{\rm F} \leq \norm{\smash{(\mat{R}^* \mat{M})^+ \mat{R}^*}}_{\rm F}^2.
\end{equation}
\end{lemma}
\begin{proof}
Introduce the orthogonal projection $\mat{\varPi}_{\mat{M}} = \mat{M} (\mat{M}^* \mat{M})^+ \mat{M}^*$ onto the range of $\mat{M}$.
Since $\mat{\varPi}_{\mat{M}}$ is an orthogonal projection, it is bounded by the identity matrix in the psd order.
\begin{equation*}
    \mat{M} (\mat{M}^* \mat{M})^+ \mat{M}^* \preccurlyeq \Id. 
\end{equation*}
The psd order is preserved under conjugation.
Thus, we may conjugate this relation by $\mat{R} \mat{R}^* \mat{M}$ to obtain
\begin{equation}
\label{eq:to_conjugate}
    (\mat{M}^* \mat{R} \mat{R}^* \mat{M}) (\mat{M}^* \mat{M})^{+} (\mat{M}^* \mat{R} \mat{R}^* \mat{M})
    \preccurlyeq  \mat{M}^* (\mat{R} \mat{R}^*)^2 \mat{M}.
\end{equation}
Since $\mat{R}$ is invertible, we know that $\range(\mat{M}^* \mat{R} \mat{R}^* \mat{M})
    = \range(\mat{M}^* \mat{R}) = \range(\mat{M}^*)$.
Therefore,
\begin{equation}
\label{eq:use_me}
    (\mat{M}^* \mat{R} \mat{R}^* \mat{M}) 
    (\mat{M}^* \mat{R} \mat{R}^* \mat{M})^+
    = \mat{\varPi}_{\mat{M}^* \mat{R} \mat{R}^* \mat{M}} 
    = \mat{\varPi}_{\mat{M}^*}.
\end{equation}
By conjugating \cref{eq:to_conjugate} by $(\mat{M}^* \mat{R} \mat{R}^* \mat{M})^+$ and using \cref{eq:use_me}, we find
\begin{equation*}
    (\mat{M}^* \mat{M})^{+}
    \preccurlyeq  
    (\mat{M}^* \mat{R} \mat{R}^* \mat{M})^+
    \mat{M}^* (\mat{R} \mat{R}^*)^2 \mat{M}
    (\mat{M}^* \mat{R} \mat{R}^* \mat{M})^+ = (\mat{R}^* \mat{M})^+ (\mat{R}^* \mat{R}) (\mat{M}^* \mat{R})^+.
\end{equation*}
Take the trace to yield
\begin{equation*}
    \norm{\smash{\mat{M}^+}}^2_{\rm F} = \tr\bigl(\mat{M}^+ (\mat{M}^*)^+\bigr)
    \leq \tr\bigl((\mat{R}^* \mat{M})^+ (\mat{R}^* \mat{R}) (\mat{M}^* \mat{R})^+\bigr) = \norm{\smash{(\mat{R}^* \mat{M})^+ \mat{R}^*}}_{\rm F}^2.
\end{equation*}
We have obtained the desired conclusion.
\end{proof}

\subsection{Proof of \texorpdfstring{\cref{thm:partial-isometry}}{Theorem 2.1}} \label{sec:partial-isometry-proof}

Let $\mat{A} = \mat{Q}\mat{R}$ be any decomposition which factors $\mat{A}$ as a product of a matrix $\mat{Q} \in \field^{n\times r}$ with orthonormal columns and a matrix $\mat{R} \in \field^{r\times d}$ with linearly independent rows, and extend $\mat{Q}$ to a square unitary matrix $[\begin{matrix} \mat{Q} & \mat{Q}_\perp \end{matrix}] \in \field^{n\times n}$.
Due to rotation invariance, 
\begin{equation*}
\mat{\varPsi} \coloneqq \twobyone{\mat{\varPsi}_1}{\mat{\varPsi}_2} = \twobyone{\mat{Q}^*}{\mat{Q}_{\perp}^*} \mat{\varOmega}
\end{equation*}
has the same distribution as $\mat{\varOmega}$, either Gaussian or random orthonormal.
In either case, 
\begin{equation*}
\rank(\mat{\varOmega}^*\mat{A})=\rank(\mat{A})=r
\end{equation*}
almost surely, and $\expect \bigl[(\mat{\varPsi}_1^*)^+ \mat{\varPsi}_2^*\bigr] = \mat{0}$.
\Cref{lem:sketch-solve-formula} implies that the difference between the sketch-and-solve solution $\Xhat = (\mat{\varOmega}^*\mat{A})^+(\mat{\varOmega}^*\mat{B})$ and the exact solution $\mat{X}_\star = \mat{A}^+\mat{B}$ is
\begin{equation*}
    \Xhat - \mat{X}_\star 
    = \mat{R}^+ (\mat{\varOmega}^*\mat{Q})^+(\mat{\varOmega}^*\mat{Q}_\perp^{\vphantom{*}}) \mat{Q}_\perp^*\mat{B}
    = \mat{R}^+ (\mat{\varPsi}_1^*)^+ \mat{\varPsi}_2^* \mat{Q}_\perp^*\mat{B},
\end{equation*}
which has expected value $\mat{0}$.
Therefore, we conclude that the sketch-and-solve solution is unbiased.
\Cref{lem:sketch-solve-formula} also implies the squared residual norm is exactly
\begin{equation*}
    \norm{\mat{B} - \mat{A}\Xhat}_{\rm F}^2
    = \norm{\smash{(\mat{\varOmega}^*\mat{Q})^+(\mat{\varOmega}^*\mat{Q}_\perp^{\vphantom{*}}) \mat{Q}_\perp^*\mat{B}}}_{\rm F}^2 + \norm{\mat{Q}_\perp^*\mat{B}}_{\rm F}^2 
    = \norm{(\mat{\varPsi}_1^*)^+ \mat{\varPsi}_2^* \mat{Q}_{\perp}^* \mat{B}}_{\rm F}^2 + \norm{\mat{Q}_\perp^*\mat{B}}_{\rm F}^2.
\end{equation*}
\Cref{prop:partial-isom,prop:gaussian-sketch} imply that
\begin{equation*}
    \begin{alignedat}{2}
        \expect \norm{(\mat{\varPsi}_1^*)^+ \mat{\varPsi}_2^* \mat{Q}_{\perp}^* \mat{B}}_{\rm F}^2 &= \frac{r}{\ell - r - \alpha_\field} \norm{\mat{Q}_\perp^*\mat{B}}_{\rm F}^2, & \quad \text{(Gaussian)} \\
        \text{or } \expect \norm{(\mat{\varPsi}_1^*)^+ \mat{\varPsi}_2^* \mat{Q}_{\perp}^* \mat{B}}_{\rm F}^2 &= 
        \frac{n-\ell}{n-r} \frac{r}{\ell - r - \alpha_\field} \norm{\mat{Q}_\perp^*\mat{B}}_{\rm F}^2. & \quad \text{(random orthonormal)}
    \end{alignedat}
\end{equation*}
Since $\norm{\mat{Q}_\perp^*\mat{B}}_{\rm F}^2 = \norm{\mat{B} - \mat{A}\mat{X}_{\star}}_{\rm F}^2$, this completes the proof.

\subsection{Proof of \texorpdfstring{\cref{thm:sketch-solve-optimality}}{Theorem 2.2}} \label{sec:proof_optimality}

Consider any inconsistent least-squares problem
\begin{equation}
\label{eq:original}
    \min_{\mat{X} \in \field^{d\times p}} \norm{\mat{B}_0 - \mat{A}_0 \mat{X}}_{\rm F}^2,
\end{equation}
involving matrices $\mat{A}_0 \in \field^{n \times d}$ and $\mat{B}_0 \in \field^{n \times p}$ with $\rank(\mat{A}_0) = r$.
Let $\mat{U}$ be an independent, random unitary matrix $\mat{U} \in \field^{n \times n}$ and set
\begin{equation*}
    \tilde{\mat{A}} = \mat{U} \mat{A}_0,
    \quad \tilde{\mat{B}} = \mat{U} \mat{B}_0
    \quad \text{and} \quad
    \tilde{\mat{\varOmega}} = \mat{U}^* \mat{\varOmega}.
\end{equation*}
Conditioning on $\mat{\varOmega}$, we have
$\rank(\tilde{\mat{\varOmega}}^*\mat{A}_0)=r$
almost surely.
If we apply sketch-and-solve with the embedding $\mat{\varOmega}$ to the random problem class
\begin{equation*}
    \min_{\mat{X} \in \field^{d\times p}} 
    \norm{\smash{\mat{\tilde{B}} - \mat{\tilde{A}}\mat{X}}}_{\rm F}^2,
\end{equation*}
this leads to the same estimator $\Xhat = (\mat{\varOmega}^* \tilde{\mat{A}})^+(\mat{\varOmega}^* \tilde{\mat{B}}) = (\tilde{\mat{\varOmega}}^* \mat{A}_0)^+(\tilde{\mat{\varOmega}}^* \mat{B}_0)$ as applying sketch-and-solve with the embedding $\tilde{\mat{\varOmega}}$ to the original problem \cref{eq:original}.

Now let $\mat{Q} \in \field^{n \times r}$ be any matrix with orthonormal columns that shares the same range as $\mat{A}_0$, and extend $\mat{Q}$ to a square unitary matrix $[\begin{matrix} \mat{Q} & \mat{Q}_\perp \end{matrix}] \in \field^{n\times n}$.
Further let $\mat{\varOmega} = \mat{\varPsi} \mat{R}$ be any decomposition which factors $\mat{\varOmega}$ as a product of a matrix $\mat{\varPsi} \in \field^{n\times \ell}$ with orthonormal columns and a full-rank matrix $\mat{R} \in \field^{\ell \times \ell}$.
Introduce the random orthonormal matrix
\begin{equation*}
    \tilde{\mat{\varPsi}} = \twobyone{\tilde{\mat{\varPsi}}_1}{\tilde{\mat{\varPsi}}_2}
    = \twobyone{\mat{Q}^*}{\mat{Q}_{\perp}^*} \mat{U}^* \mat{\varPsi} \in \field^{n \times \ell}.
\end{equation*}
\Cref{lem:sketch-solve-formula} implies the squared residual norm is exactly
\begin{align*}
    \norm{\smash{\mat{\tilde{B}} - \mat{\tilde{A}}\Xhat}}_{\rm F}^2 
    &= \norm{\mat{B}_0 - \mat{A}_0\Xhat}_{\rm F}^2 \\
    &= \norm{\smash{(\tilde{\mat{\varOmega}}^*\mat{Q})^+(\tilde{\mat{\varOmega}}^*\mat{Q}_\perp^{\vphantom{*}}) \mat{Q}_\perp^*\mat{B}_0}}_{\rm F}^2 + \norm{\mat{Q}_\perp^*\mat{B}_0}_{\rm F}^2 \\
    &= \norm{\smash{(\mat{R}^* \tilde{\mat{\varPsi}}_1^*)^+(\mat{R}^* \tilde{\mat{\varPsi}}_2^*) \mat{Q}_\perp^*\mat{B}_0}}_{\rm F}^2 + \norm{\mat{Q}_\perp^*\mat{B}_0}_{\rm F}^2.
\end{align*}
\Cref{prop:partial-isom} then implies that
\begin{equation*}
    \expect \norm{\smash{(\mat{R}^* \tilde{\mat{\varPsi}}_1^*)^+(\mat{R}^* \tilde{\mat{\varPsi}}_2^*) \mat{Q}_\perp^*\mat{B}_0}}_{\rm F}^2 
    \geq \frac{n-\ell}{n-r} \, \frac{r}{\ell - r - \alpha_\field} \, \norm{\mat{Q}_\perp^*\mat{B}_0}_{\rm F}^2.
\end{equation*}
Since $\norm{\mat{Q}_\perp^*\mat{B}_0}_{\rm F}^2 = \norm{\smash{\tilde{\mat{B}} - \tilde{\mat{A}} \tilde{\mat{A}}^+ \tilde{\mat{B}}}}_{\rm F}^2$, we conclude that the expected error is
\begin{equation*}
    \frac{\expect \norm{\smash{\tilde{\mat{B}} - \tilde{\mat{A}} (\mat{\varOmega}^* \tilde{\mat{A}})^+ (\mat{\varOmega}^* \tilde{\mat{B}})}}_{\rm F}^2}
    {\norm{\smash{\tilde{\mat{B}} - \tilde{\mat{A}} \tilde{\mat{A}}^+ \tilde{\mat{B}}}}_{\rm F}^2}
    \geq 1 + \frac{n-\ell}{n-r} \, \frac{r}{\ell - r - \alpha_\field}.
\end{equation*}
Last, by the probabilistic method, the worst-case matrices $\mat{A},\mat{B}$ must lead to at least the level of error incurred by the average-case random matrices $\mat{\tilde{A}},\mat{\tilde{B}}$.
\begin{equation*}
    \sup_{\substack{\mat{A}\in \field^{n\times d},\: \mat{B} \in \field^{n\times p} \\
    \rank(\mat{A}) = r,\,\mat{B}\neq\mat{A}\mat{A}^+\mat{B}}} \frac{\expect \norm{\smash{\mat{B} - \mat{A}(\mat{\varOmega}^*\mat{A})^+ (\mat{\varOmega}^*\mat{B})}}_{\rm F}^2}{\norm{\smash{\mat{B} - \mat{A} \mat{A}^+ \mat{B}}}_{\rm F}^2} 
    \ge \frac{\expect \norm{\smash{\tilde{\mat{B}} - \tilde{\mat{A}} (\mat{\varOmega}^* \tilde{\mat{A}})^+ (\mat{\varOmega}^* \tilde{\mat{B}})}}_{\rm F}^2}
    {\norm{\smash{\tilde{\mat{B}} - \tilde{\mat{A}} \tilde{\mat{A}}^+ \tilde{\mat{B}}}}_{\rm F}^2}.
\end{equation*}
This completes the proof.

\section{Analysis of low-rank approximations} \label{sec:rsvd-nystom-analysis}

This section analyzes algorithms for low-rank approximation.
\Cref{sec:schur} presents tools from the theory of Schur complements, and \cref{sec:hard_instance} introduces a hard problem class where the Schur complement is large and low-rank approximation struggles.
Using properties of Schur complements, \cref{sec:proof_rsvd} shows that this hard example upper bounds the error of low-rank approximation algorithms with rotation-invariant embeddings applied to all matrices, thus establishing an upper bound for the Nystr\"om approximation and randomized SVD.
Last,
\cref{sec:proof_rsvd_lower} establishes a matching lower bound for the Nystr\"om approximation and randomized SVD, and \cref{sec:proof_gen_nystrom_lower} establishes a matching lower bound for the generalized Nystr\"om approximation.

\subsection{Schur complements}
\label{sec:schur}

The residual of the Nystr\"om approximation is the \emph{Schur complement}
\begin{equation*}
    \mat{H} / \mat{\varOmega} = \mat{H} - \mat{H}\langle \mat{\varOmega}\rangle.
\end{equation*}
Here, we list a few properties of Nystr\"om approximation and Schur complement that we will need for our analysis; see \cite[sec.~1.5]{Bha07} and \cite{And05}.

\begin{fact}[Properties of Nystr\"om approximation and Schur complement] \label{fact:nystrom-properties}
    Let $\mat{\varOmega} \in \field^{n\times \ell}$ be an embedding, and let $\mat{H} \in \field^{n\times n}$ be a psd matrix.
    Then the Nystr\"om approximation and Schur complement satisfy the following properties. 
    \begin{enumerate}[label=(\alph*)]
        \item \textbf{Psd.}
        The Nystr\"om approximation $\mat{H}\langle \mat{\varOmega}\rangle$ and the Schur complement $\mat{H} / \mat{\varOmega}$ are psd. \label{item:nystrom-positive}     
        \item \textbf{Representation invariance.}
        For any nonsingular matrix $\mat{M}$, it holds that $\mat{H}\langle \mat{\varOmega}\mat{M}\rangle = \mat{H}\langle \mat{\varOmega}\rangle$. \label{item:nystrom-invariance}
        \item \label{item:nystrom-concave} \textbf{Concavity.}
        The map $\mat{H} \mapsto \mat{H} / \mat{\varOmega}$ is concave with respect to the psd order $\preccurlyeq$. If $\mat{H}_1,\ldots,\mat{H}_t$ are psd matrices and $\theta_1,\ldots,\theta_t$ are nonnegative parameters summing to one, then
        \begin{equation*}
            (\theta_1 \mat{H}_1 + \theta_2 \mat{H}_2 + \cdots + \theta_t \mat{H}_t)/ \mat{\varOmega} \succcurlyeq \theta_1 (\mat{H}_1/\mat{\varOmega}) + \theta_2 (\mat{H}_2/\mat{\varOmega}) + \cdots + \theta_t (\mat{H}_t / \mat{\varOmega}).
        \end{equation*}
        \item  \label{item:nystrom-monotone} \textbf{Monotonicity.}
        The map $\mat{H} \mapsto \mat{H} / \mat{\varOmega}$ is monotone with respect to the psd order $\preccurlyeq$,
        \begin{equation*}
            \mat{H}_1 \preccurlyeq \mat{H}_2 \implies \mat{H}_1 / \mat{\varOmega} \preccurlyeq \mat{H}_2 / \mat{\varOmega}.
        \end{equation*}
    \end{enumerate}
\end{fact}

\subsection{Analysis of a hard instance} \label{sec:hard_instance}

This section calculates the error of the Nystr\"om approximation applied to a matrix with two distinct nonzero eigenvalues as the top eigenvalue is raised to infinity.

\begin{lemma}[Nystr\"om approximation of a matrix with two eigenvalues] \label{lem:two-eigenvalues}
    Let $\mat{\varOmega} \in \field^{n \times \ell}$ be any full-rank random embedding that satisfies $\mat{\varOmega} \sim \mat{U} \mat{\varOmega}$ for each unitary $\mat{U} \in \field^{n \times n}$.
    For any positive numbers $a,b > 0$ and any rank parameters $q, r \in \mathbb{N}$ with $q + \alpha_{\field} < \ell \leq r \leq n$, define a psd matrix
    \begin{equation*}
        \mat{H}_{a,b} = \begin{bmatrix}
            a\,\Id & & \\
            & b\,\Id & \\
            & & \mat{0}
        \end{bmatrix},
    \end{equation*}
    that is partitioned into the first $q$, the middle $r - q$, and the last $n - r$ coordinates.
    Then it follows
    \begin{equation*}
        \lim_{a \rightarrow \infty} \expect \tr\bigl(\mat{H}_{a,b} - \mat{H}_{a,b}\langle \mat{\varOmega} \rangle\bigr) 
        = b(r-\ell) \biggl(1 + \frac{q}{\ell - q - \alpha_\field}\biggr).
    \end{equation*}
\end{lemma}

\begin{proof}
    Our first step reduces the dimensionality of $\mat{H}_{a,b}$ to $r \times r$.
    We write 
    \begin{equation*}
    \mat{\varOmega} = \begin{bmatrix} \mat{\varOmega}_1 \\ \mat{\varOmega}_2 \\ \mat{\varOmega}_3 \end{bmatrix},
    \end{equation*} 
    where we have partitioned the first $q$, the middle $r-q$, and the last $n-r$ rows of $\mat{\varOmega}$, and we observe that
    \begin{equation*}
        \mat{H}_{a,b}\langle \mat{\varOmega} \rangle 
        = \begin{bmatrix} a \mat{\varOmega}_1 \\ b \mat{\varOmega}_2 \\ \mat{0}
        \end{bmatrix}
        (a\,\mat{\varOmega}_1^* \mat{\varOmega}_1^{\vphantom{*}} + b\,\mat{\varOmega}_2^* \mat{\varOmega}_2^{\vphantom{*}})^{-1}
        \begin{bmatrix} a \mat{\varOmega}_1^* & b \mat{\varOmega}_2^* & \mat{0} \end{bmatrix}
        = \begin{bmatrix} \begin{bmatrix} a\,\Id & \mat{0} \\
        \mat{0} & b\,\Id
        \end{bmatrix} \left\langle \begin{bmatrix} \mat{\varOmega}_1 \\ \mat{\varOmega}_2 \end{bmatrix} \right\rangle & \mat{0} \\
        \mat{0} & \mat{0}
        \end{bmatrix}.
    \end{equation*}
    Next we introduce the random orthonormal matrix 
    \begin{equation*}
    \mat{\varPsi} = \begin{bmatrix} \mat{\varPsi}_1 \\ \mat{\varPsi}_2 \end{bmatrix}
    = \begin{bmatrix} \mat{\varOmega}_1 \\ \mat{\varOmega}_2 \end{bmatrix} \bigl(\mat{\varOmega}_1^* \mat{\varOmega}_1^{\vphantom{*}} + \mat{\varOmega}_2^* \mat{\varOmega}_2^{\vphantom{*}}\bigr)^{-1/2} \in \field^{r \times \ell}.
    \end{equation*}
    We also take a QR decomposition
    \begin{equation*}
    \mat{\varPsi}^* = \begin{bmatrix} \mat{\varPsi}_1^* & \mat{\varPsi}_2^* \end{bmatrix}
    = \mat{Q} \mat{R}
    = \mat{Q} \begin{bmatrix} \mat{R}_1 & \mat{R}_2 \\ \mat{0} & \mat{R}_{\perp} \end{bmatrix} \in \field^{\ell \times r},
    \end{equation*}
    where $\mat{Q} \in \field^{\ell \times \ell}$ is unitary and $\mat{R} \in \field^{\ell \times r}$ is upper triangular.
    Since $\mat{\varPsi}$ and $\mat{Q}$ have orthonormal columns, $\mat{R}$ must have orthonormal rows and therefore
    \begin{equation}
    \label{eq:orth_rows}
    	\begin{bmatrix} \Id & \mat{0} \\ \mat{0} & \Id \end{bmatrix}
	= \begin{bmatrix} \mat{R}_1 & \mat{R}_2 \\ \mat{0} & \mat{R}_{\perp} \end{bmatrix}
	\begin{bmatrix} \mat{R}_1^* & \mat{0} \\
	\mat{R}_2^* & \mat{R}_{\perp}^* \end{bmatrix}
	= \begin{bmatrix} \mat{R}_1^{\vphantom{*}} \mat{R}_1^* + \mat{R}_2^{\vphantom{*}} \mat{R}_2^* & \mat{R}_2^{\vphantom{*}} \mat{R}_{\perp}^* \\ \mat{R}_{\perp} \mat{R}_2^* & \mat{R}_{\perp}^{\vphantom{*}} \mat{R}_{\perp}^* \end{bmatrix}.    \end{equation}
    By \cref{fact:nystrom-properties}\ref{item:nystrom-invariance}, we observe that
    \begin{equation*}
        \begin{bmatrix} a\,\Id & \mat{0} \\
        \mat{0} & b\,\Id
        \end{bmatrix} \left\langle \begin{bmatrix} \mat{\varOmega}_1 \\ \mat{\varOmega}_2 \end{bmatrix} \right\rangle
        = \begin{bmatrix} a\,\Id & \mat{0} \\
        \mat{0} & b\,\Id
        \end{bmatrix} \langle \mat{\varPsi} \rangle
        = \begin{bmatrix} a\,\Id & \mat{0} \\
        \mat{0} & b\,\Id
        \end{bmatrix} \langle \mat{R}^* \rangle.
    \end{equation*}
    Now we make an explicit calculation using \cref{eq:orth_rows}.
    \begin{align*}
         \begin{bmatrix} a\,\Id & \mat{0} \\
        \mat{0} & b\,\Id
        \end{bmatrix} \langle \mat{R}^* \rangle
        &= \begin{bmatrix}
        a \mat{R}^*_1 & \mat{0} \\
        b \mat{R}^*_2 & b \mat{R}^*_{\perp}
        \end{bmatrix}
        \begin{bmatrix}
            a \mat{R}_1 \mat{R}_1^* + b \mat{R}_2 \mat{R}_2^* & \mat{0} \\
            \mat{0} & b \Id
        \end{bmatrix}^{-1}
        \begin{bmatrix}
        a \mat{R}_1 & b \mat{R}_2
        \\ \mat{0} & b \mat{R}_{\perp}
        \end{bmatrix} \\
        &= \begin{bmatrix} a \mat{R}_1^* \\
        b \mat{R}_2^* \end{bmatrix}
        (a \mat{R}_1^{\vphantom{*}} \mat{R}_1^* + b \mat{R}_2^{\vphantom{*}} \mat{R}^*_2)^{-1}
        \begin{bmatrix} a \mat{R}_1 &
        b \mat{R}_2 \end{bmatrix}
        + \begin{bmatrix} \mat{0} & \mat{0} \\
        \mat{0} & b \mat{R}_{\perp}^* \mat{R}_{\perp}^{\vphantom{*}}
        \end{bmatrix}.
    \end{align*}
    For $a > b$, another explicit calculation using \cref{eq:orth_rows} yields
    \begin{align*}
        & \tr\Biggl(\begin{bmatrix} a\,\Id & \mat{0} \\
        \mat{0} & b\,\Id
        \end{bmatrix}
        - \begin{bmatrix} a\,\Id & \mat{0} \\
        \mat{0} & b\,\Id
        \end{bmatrix} \langle \mat{R}^* \rangle\Biggr) \\
        &= aq + b(r-q) - \tr\bigl( (a^2 \mat{R}_1^{\vphantom{*}} \mat{R}_1^* + b^2 \mat{R}_2^{\vphantom{*}} \mat{R}_2^*) (a \mat{R}_1^{\vphantom{*}} \mat{R}_1^* + b \mat{R}_2^{\vphantom{*}} \mat{R}_2^*)^{-1} \bigr) - b(\ell - q) \\
        &= (a-b) \tr\bigl( b\mat{R}_2^{\vphantom{*}} \mat{R}_2^* (a \mat{R}_1^{\vphantom{*}} \mat{R}_1^* + b \mat{R}_2^{\vphantom{*}} \mat{R}_2^*)^{-1}\bigr) + b(r-\ell) \\
        &= \tr\bigl( b\mat{R}_2^{\vphantom{*}} \mat{R}_2^* (\mat{R}_1^{\vphantom{*}} \mat{R}_1^* + \tfrac{b}{a-b} \mat{R}_2^{\vphantom{*}} \mat{R}_2^*)^{-1}\bigr) + b(r-\ell).
    \end{align*}
    This expression for the Nystr\"om approximation error simplifies as we take the limit $a \rightarrow \infty$.
    \begin{align*}
        \lim_{a \rightarrow \infty} \tr\bigl( b\mat{R}_2^{\vphantom{*}} \mat{R}_2^* (\mat{R}_1^{\vphantom{*}} \mat{R}_1^* + \tfrac{b}{a-b} \mat{R}_2^{\vphantom{*}} \mat{R}_2^*)^{-1}\bigr) + b(r-\ell)
        &= b \tr \bigl(\mat{R}_2^{\vphantom{*}} \mat{R}_2^* (\mat{R}_1^{\vphantom{*}} \mat{R}_1^*)^{-1}\bigr) + b(r-\ell) \\
        &= b \tr\bigl( (\mat{R}_1^{\vphantom{*}} \mat{R}_1^*)^{-1}\bigr) + b(r-\ell - q).
    \end{align*}
    Finally, apply the monotone convergence theorem and use the formula for $\expect \tr(\mat{R}_1^*\mat{R}_1^{\vphantom{*}})^{-1}$ from \cref{prop:beta_exp} to yield
    \begin{align*}
    	\lim_{a \rightarrow \infty} \expect \tr\bigl(\mat{H}_{a,b} - \mat{H}_{a,b}\langle \mat{\varOmega} \rangle\bigr)
        &= b \expect \tr(\mat{R}_1^{\vphantom{*}} \mat{R}_1^*)^{-1} + b(r-\ell-q) \\
        &= b q \left(1 + \frac{r -\ell}{\ell - q - \alpha_\field}\right) + b(r-\ell-q)
        = b(r-\ell) \biggl(1 + \frac{q}{\ell - q - \alpha_\field}\biggr).
    \end{align*}
    This completes the proof.
\end{proof}

\subsection{Proof of \texorpdfstring{\cref{thm:rsvd-upper-bound}}{Theorem 3.2} and  \texorpdfstring{\cref{thm:nystrom}\ref{item:nystrom-result-upper}}{Theorem 3.6(a)}} \label{sec:proof_rsvd}

Introduce an eigendecomposition
\begin{equation*}
    \mat{H} = \mat{Q} \mat{\varLambda} \mat{Q}^*,
\end{equation*}
where $\mat{Q} \in \field^{n \times n}$ is unitary and $\mat{\varLambda} = \diag(\lambda_1,\ldots,\lambda_r,0,\ldots,0)$ for $\lambda_1 \geq \cdots \geq \lambda_r > 0$.
By rotational invariance, it holds that $\mat{Q}^* \mat{\varOmega} \sim \mat{\varOmega}$ and thus
\begin{equation*}
    (\mat{Q} \mat{\varLambda} \mat{Q}^*)\langle \mat{\varOmega} \rangle
    = \mat{Q} \bigl(\mat{\varLambda}\langle \mat{Q}^*\mat{\varOmega}\rangle\bigr) \mat{Q}^*
    \sim \mat{Q} \bigl(\mat{\varLambda}\langle \mat{\varOmega} \rangle\bigr) \mat{Q}^*.
\end{equation*}
Similarly, if $\mat{P} \in \{0, 1\}^{n \times n}$ is a permutation matrix, independent from $\mat{\varOmega}$, then $\mat{P}^* \mat{\varOmega} \sim \mat{\varOmega}$ and thus
\begin{equation*}
    \mat{P}^* \bigl(\mat{\varLambda}\langle \mat{\varOmega} \rangle \bigr) \mat{P}
    = \bigl(\mat{P}^* \mat{\varLambda} \mat{P}\bigr) \langle \mat{P}^* \mat{\varOmega} \rangle
    \sim \bigl(\mat{P}^* \mat{\varLambda} \mat{P}\bigr) \langle \mat{\varOmega} \rangle.
\end{equation*}
Consequently, the Schur complements satisfy
\begin{equation*}
    \expect \tr\bigl(\mat{H} / \mat{\varOmega} \bigr)
    = \expect \tr\bigl(\mat{\varLambda} / \mat{\varOmega} \bigr)
    = \expect \tr\bigl((\mat{P}^* \mat{\varLambda} \mat{P})
    / \mat{\varOmega}\bigr).
\end{equation*}
Thus we can permute the diagonal entries however we want without changing the expectation.
Specifically, we can introduce a permutation $\mat{P} \in \{0, 1\}^{n \times n}$ that permutes the first $q$ coordinates, permutes the middle $r-q$ coordinates, and permutes the last $n - r$ coordinates, each uniformly at random.
It follows that $\expect \bigl[\mat{P}^* \mat{\varLambda} \mat{P}\bigr] = \mat{H}_{a_0, b_0}$, where
\begin{equation*}
    \mat{H}_{a,b}
    = \begin{bmatrix}
        a\,\Id & & \\
        & b\,\Id & \\
        & & \mat{0}
    \end{bmatrix},
    \quad
    a_0 = \frac{1}{q} \sum_{i=1}^q \lambda_i = \frac{\tr(\lowrank{\mat{H}}_q)}{q} 
    \quad \text{and}
    \quad
    b_0
    = \frac{1}{r-q} \sum_{i=q+1}^r \lambda_i
    = \frac{\tr(\mat{H} - \lowrank{\mat{H}}_q)}{r-q}.
\end{equation*}
We next use the concavity of the Schur complement (\cref{fact:nystrom-properties}\ref{item:nystrom-concave}) to compute
\begin{equation*}
    \expect\bigl[(\mat{P}^* \mat{\varLambda} \mat{P})
    / \mat{\varOmega}\,\big|\, \mat{\varOmega}\bigr]
    \preceq
    \expect\bigl[\mat{P}^* \mat{\varLambda} \mat{P}\bigr]
    / \mat{\varOmega}.
\end{equation*}
Last, we take the expected trace, use the monotonicity of the Schur complement (\cref{fact:nystrom-properties}\ref{item:nystrom-monotone}), and apply the analysis of $\mat{H}_{a,b}$ from \cref{lem:two-eigenvalues}.
\begin{align*}
    \expect \tr\bigl(\mat{H}/\mat{\varOmega}\bigr)
    &\leq \expect \tr\bigl(\mat{H}_{a_0, b_0} / \mat{\varOmega}\bigr) \\
    &\leq \lim_{a \rightarrow \infty} \expect \tr\bigl(\mat{H}_{a, b_0} / \mat{\varOmega}\bigr) \\
    &= b_0(r-\ell) \biggl(1 + \frac{q}{\ell - q - \alpha_\field}\biggr) \\
    &= \frac{r-\ell}{r-q} \left(1 + \frac{q}{\ell-q-\alpha_\field} \right) \tr\bigl(\mat{H} - \lowrank{\mat{H}}_q\bigr).
\end{align*}
We have proved \cref{thm:nystrom}\ref{item:nystrom-result-upper}.
By invoking the Gram correspondence (\cref{fact:gram-correspondence}), \cref{thm:rsvd-upper-bound} follows.

\subsection{Proof of \texorpdfstring{\cref{thm:rsvd-lower-bound}}{Theorem 3.3} and \texorpdfstring{\cref{thm:nystrom}\ref{item:nystrom-result-lower}}{Theorem 3.6(b)}} \label{sec:proof_rsvd_lower}

For each positive number $a > 0$, introduce the psd matrix
\begin{equation*}
    \mat{H}_a = \begin{bmatrix}
        a\,\Id & & \\
        & \Id & \\
        & & \mat{0}
    \end{bmatrix}
\end{equation*}
which is partitioned into the first $q$, the middle $r - q$, and the last $n - r$ coordinates.
Further, introduce a random unitary matrix $\mat{U} \in \field^{n\times n}$, independent from $\mat{\varOmega}$, and set
\begin{equation*}
    \mat{\tilde{H}}_a = \mat{U} \mat{H}_a \mat{U}^*
    \quad \text{and} \quad \mat{\tilde{\varOmega}} = \mat{U}^* \mat{\varOmega}.
\end{equation*}
We will study the error from applying the randomized Nystr\"om approximation with the embedding $\mat{\varOmega}$ to the random matrix $\mat{\tilde{H}}_a$.
Notice the Schur complement satisfies
\begin{equation*}
    \tr\bigl(\mat{\tilde{H}}_a / \mat{\varOmega}\bigr)
    = \tr\bigl( \mat{H}_a / \mat{\tilde{\varOmega}} \bigr).
\end{equation*}
\Cref{lem:two-eigenvalues} implies that
\begin{equation*}
    \lim_{a \rightarrow \infty} 
    \frac{\expect \tr\bigl(\mat{\tilde{H}}_a / \mat{\varOmega}\bigr)}{\tr\bigl(\mat{H}_a - \lowrank{\mat{H}_a}_q\bigr)} 
    = \lim_{a \rightarrow \infty} 
    \frac{\expect \tr\bigl(\mat{H}_a / \mat{\tilde{\varOmega}}\bigr)}{\tr\bigl(\mat{H}_a - \lowrank{\mat{H}_a}_q\bigr)} 
    = \frac{r-\ell}{r-q}\, \left(1 + \frac{q}{\ell - q - \alpha_\field}\right).
\end{equation*}
Last, by the probabilistic method, the worst-case matrix $\mat{H}$ must lead to at least the level of error incurred by the average-case random matrix $\mat{H}_a$ for any $a > 0$.
\begin{equation*}
    \sup_{\substack{\mat{H} \in \field^{n\times n} \text{ psd} \\ \rank(\mat{H}) = r}} \frac{\expect \tr(\mat{H} / \mat{\varOmega})}{\tr(\mat{H} - \lowrank{\mat{H}}_q)}
    \geq \lim_{a \rightarrow \infty} 
    \frac{\expect \tr\bigl(\mat{\tilde{H}}_a / \mat{\varOmega}\bigr)}{\tr\bigl(\mat{H}_a - \lowrank{\mat{H}_a}_q\bigr)} .
\end{equation*}
We have proved \cref{thm:nystrom}\ref{item:nystrom-result-lower}.
By invoking the Gram correspondence (\cref{fact:gram-correspondence}), \cref{thm:rsvd-lower-bound} follows.

\subsection{Proof of \texorpdfstring{\cref{thm:gen_nystrom}\ref{item:gen-nystrom-result-lower}}{Theorem 4.2(b)}} \label{sec:proof_gen_nystrom_lower}

For each positive number $a > 0$, introduce the $d \times n$ matrix
\begin{equation*}
    \mat{H}_a = \begin{bmatrix}
        a\,\Id & & \\
        & \Id & \\
        & & \mat{0}
    \end{bmatrix}
\end{equation*}
which is partitioned into the first $q$, the middle $r - q$, and the last $\min\{d,n\} - r$ coordinates.
Let $\mat{U}\in\field^{d\times d}$ and
$\mat{V}\in\field^{n\times n}$ be independent random unitary matrices,
independent of $(\mat{\varOmega},\mat{\varPsi})$, and define
\begin{equation*}
    \tilde{\mat{H}}_a=\mat{U}\mat{H}_a\mat{V}^*,
    \qquad
    \tilde{\mat{\varOmega}}=\mat{V}^*\mat{\varOmega},
    \qquad
    \tilde{\mat{\varPsi}}=\mat{U}^*\mat{\varPsi}.
\end{equation*}
Conditioning on $(\mat{\varOmega},\mat{\varPsi})$, we have
\begin{equation*}
    \rank(\mat{H}_a\tilde{\mat{\varOmega}}) =
    \rank(\tilde{\mat{\varPsi}}^*\mat{H}_a\tilde{\mat{\varOmega}})
    = \ell
    \quad\text{almost surely}.
\end{equation*}
We will study the error from applying the generalized Nystr\"om approximation with the embeddings $\mat{\varOmega}$ and $\mat{\varPsi}$ to the random matrix $\mat{\tilde{H}}_a$.
First, by unitary invariance of the Frobenius norm,
\begin{equation*}
    \norm{\tilde{\mat{H}}_a
    - \tilde{\mat{H}}_a\langle \mat{\varOmega}, \mat{\varPsi} \rangle}_{\rm F}^2
    = \norm{\mat{H}_a
    - \mat{H}_a\langle \tilde{\mat{\varOmega}}, \tilde{\mat{\varPsi}} \rangle}_{\rm F}^2.
\end{equation*}
Next we view $\norm{\mat{H}_a
- \mat{H}_a\langle \tilde{\mat{\varOmega}}, \tilde{\mat{\varPsi}} \rangle}_{\rm F}^2$ as the squared error from applying sketch-and-solve with the random embedding $\tilde{\mat{\varPsi}}$ to the problem
\begin{equation*}
    \min_{\mat{X} \in \field^{\ell \times n}} \norm{\mat{H}_a
    - \bigl(\mat{H}_a \tilde{\mat{\varOmega}}\bigr) \mat{X}}_{\rm F}^2.
\end{equation*}
The proof of the sketch-and-solve lower bound in \cref{sec:proof_optimality} shows that
\begin{equation*}
    \expect \norm{\mat{H}_a
    - \mat{H}_a\langle \tilde{\mat{\varOmega}}, \tilde{\mat{\varPsi}} \rangle}_{\rm F}^2
    \geq \left(1 + \frac{d-k}{d-\ell} \frac{\ell}{k - \ell - \alpha_\field} \right) \expect \norm{\mat{H}_a
    - \bigl(\mat{H}_a \tilde{\mat{\varOmega}}\bigr)
    \bigl(\mat{H}_a \tilde{\mat{\varOmega}}\bigr)^+ \mat{H}_a }_{\rm F}^2.
\end{equation*}
Next rewrite the error in the form of a Nystr\"om approximation
\begin{equation*}
    \norm{\mat{H}_a
    - \bigl(\mat{H}_a \tilde{\mat{\varOmega}}\bigr)
    \bigl(\mat{H}_a \tilde{\mat{\varOmega}}\bigr)^+ \mat{H}_a }_{\rm F}^2
    = \tr\bigl((\mat{H}_a^* \mat{H}_a^{\vphantom{*}}) / \tilde{\mat{\varOmega}}\bigr).
\end{equation*}
The proof of the randomized Nystr\"om lower bound in \cref{sec:proof_rsvd_lower} shows that
\begin{equation*}
    \lim_{a \rightarrow \infty} \expect\Biggl[ \frac{\tr\bigl((\mat{H}_a^* \mat{H}_a^{\vphantom{*}}) / \tilde{\mat{\varOmega}}\bigr)}{\tr\bigl(\mat{H}_a^* \mat{H}_a^{\vphantom{*}} - \lowrank{\smash{\mat{H}_a^* \mat{H}_a^{\vphantom{*}}}}_q\bigr)}\Biggr] = \frac{r-\ell}{r-q}\, \left(1 + \frac{q}{\ell - q - \alpha_\field}\right).
\end{equation*}
Since $\tr\bigl(\mat{H}_a^* \mat{H}_a - \lowrank{\smash{\mat{H}_a^* \mat{H}_a}}_q\bigr) = \norm{\smash{\tilde{\mat{H}}_a - \lowrank{\smash{\tilde{\mat{H}}_a}}_q}}_{\rm F}^2$, we have shown that
\begin{equation*}
    \limsup_{a \rightarrow \infty} 
    \frac{\expect \norm{\smash{\tilde{\mat{H}}_a
    - \tilde{\mat{H}}_a\langle \mat{\varOmega}, \mat{\varPsi} \rangle}}_{\rm F}^2}
    {\norm{\smash{\tilde{\mat{H}}_a - \lowrank{\smash{\tilde{\mat{H}}_a}}_q}}_{\rm F}^2}
    \geq \left(1 + \frac{d-k}{d-\ell} \frac{\ell}{k - \ell - \alpha_\field} \right)
    \frac{r-\ell}{r-q} \left( 1 + \frac{q}{\ell-q-\alpha_\field} \right).
\end{equation*}
Last, by the probabilistic method, the worst-case matrix $\mat{A}$ must lead to at least the level of error incurred by the average-case random matrix $\tilde{\mat{H}}_a$ for any $a > 0$.
\begin{equation*}
    \sup_{\substack{\mat{A} \in \field^{d\times n} \\ \rank(\mat{A}) = r}} 
    \frac{\expect \norm{\smash{\mat{A} - \mat{A}\langle\mat{\varOmega},\mat{\varPsi}\rangle}}_{\rm F}^2}
    {\norm{\smash{\mat{A} - \lowrank{\mat{A}}_q}}_{\rm F}^2} 
    \geq \limsup_{a \rightarrow \infty} \frac{\expect \norm{\smash{\tilde{\mat{H}}_a
    - \tilde{\mat{H}}_a\langle \mat{\varOmega}, \mat{\varPsi} \rangle}}_{\rm F}^2}
    {\norm{\smash{\tilde{\mat{H}}_a - \lowrank{\smash{\tilde{\mat{H}}_a}}_q}}_{\rm F}^2}.
\end{equation*}
We have proved \cref{thm:gen_nystrom}\ref{item:gen-nystrom-result-lower}.

\section{Conclusions and open problems} \label{sec:conclusion}

This work confirms the widespread belief that a random orthonormal matrix is the optimal embedding for randomized matrix approximations.
First, we have established sharp upper bounds for sketch-and-solve and several low-rank approximations using a random orthonormal matrix.
Next, we have obtained matching lower bounds to show that no other embedding can improve the worst-case theoretical performance of these algorithms.

We also submit a stronger conjecture that
no algorithm based on a random embedding
can improve on
sketch-and-solve or low-rank approximation algorithms with a random orthonormal embedding.

\begin{conjecture}[Optimality of standard randomized algorithms]
    Let $\mat{\varOmega} \in \field^{n \times \ell}$ be any full-rank random embedding.
    Then the following optimality results hold.
    \begin{enumerate}[label=(\alph*)]
        \item \textbf{Sketch-and-solve.} For any rank parameter $r \in \mathbb{N}$ with $r + \alpha_\field < \ell$, any problem dimensions $d, p \in \mathbb{N}$ with $d \geq r$, and any algorithm $\mathrm{ALG} : \field^{n \times \ell} \times \field^{\ell \times d} \times \field^{\ell \times p} \to \field^{d \times p}$,
        \begin{equation*}
            \sup_{\substack{\mat{A}\in \field^{n\times d},\: \mat{B} \in \field^{n\times p} \\
            \rank(\mat{A}) = r,\,\mat{B}\neq\mat{A}\mat{A}^+\mat{B}}} \frac{\expect \norm{\mat{B} - \mat{A}\cdot\mathrm{ALG}(\mat{\varOmega},\mat{\varOmega}^*\mat{A},\mat{\varOmega}^* \mat{B})}_{\rm F}^2}{\norm{\smash{\mat{B} - \mat{A} \mat{A}^+ \mat{B}}}_{\rm F}^2} 
            \ge 1 + \frac{n-\ell}{n-r}\, \frac{r}{\ell - r - \alpha_\field}.
        \end{equation*}
        \item \textbf{Randomized SVD.} 
        For any rank parameters $q,r \in \mathbb{N}$ with $q + \alpha_{\field} < \ell \leq r \leq n$,
        any problem dimension $d \in \mathbb{N}$ with $d \geq r$, and any algorithm $\mathrm{ALG} : \field^{n \times \ell} \times \field^{d \times \ell} \times \field^{n \times \ell} \to \field^{d \times n}$,
        \begin{equation*}
            \sup_{\substack{\mat{A} \in \field^{d \times n} \\ \rank(\mat{A}) = r}} \frac{\expect \norm{\mat{A} - \mathrm{ALG}(\mat{\varOmega},\mat{A}\mat{\varOmega},\mat{A}^*\mat{A}\mat{\varOmega})}_{\rm F}^2}
            {\norm{\smash{\mat{A} - \lowrank{\mat{A}}_q}}_{\rm F}^2} 
            \ge \frac{r -\ell}{r - q} \left(1 + \frac{q}{\ell - q - \alpha_\field}\right).
        \end{equation*}
        \item \textbf{Randomized Nystr\"om.} 
        For any rank parameters $q, r \in \mathbb{N}$ with $q + \alpha_{\field} < \ell \leq r \leq n$ and any algorithm $\mathrm{ALG} : \field^{n \times \ell}\times \field^{n \times \ell} \to \field^{n\times n}$,
        \begin{equation*}
            \sup_{\substack{\mat{H} \in \field^{n\times n} \text{ psd} \\ \rank(\mat{H}) = r}} \frac{\expect \norm{\mat{H} - \mathrm{ALG}(\mat{\varOmega},\mat{H}\mat{\varOmega})}_*}{\norm{\smash{\mat{H} - \lowrank{\mat{H}}_q}}_*} 
            \ge \frac{r-\ell}{r-q}\, \left(1 + \frac{q}{\ell - q - \alpha_\field}\right),
        \end{equation*}
        where $\norm{\cdot}_*$ is the nuclear norm.
    \end{enumerate}
\end{conjecture}
We make the opposite conjecture for generalized Nystr\"om approximation---we believe that the algorithm is not optimal in the worst-case.

Last, we connect our research with the ongoing mission to understand and justify the behavior of fast, structured random embeddings.
Here we have established rigorous bounds on the best possible performance that an embedding can achieve.
Our empirical results suggest that many random embeddings in common use nearly meet these limits, achieving accuracy similar to either a Gaussian embedding or a random orthonormal matrix.
As the sole exception in our experiments, the sign, SparseStack, and SparseIID embeddings sometimes exhibited suboptimal behavior in randomized SVD applications when the embedding dimension was $\ell < 20$ (\cref{sec:rsvd-discussion}).
More theory is needed to fully justify these empirical observations.

\subsection*{Acknowledgements}

We thank Micha\l\ Derezi\'nski, Raphael Meyer, Cameron Musco, and Christopher Musco for helpful discussions.
ENE is supported by a Miller Research Fellowship through the Miller Institute for Basic Research in Science, University
of California Berkeley.

\subsection*{AI Statement}

AI tools were useful to the authors in researching the matrix beta distribution and in proofreading the manuscript to identify typos.
The mathematical proofs and writing were done by the authors, who take full responsibility for the paper and its content.

\scriptsize
\bibliographystyle{halpha}
\let\oldthebibliography=\thebibliography
\let\endoldthebibliography=\endthebibliography
\renewenvironment{thebibliography}[1]%
  {\begin{oldthebibliography}{#1}%
   \setlength{\itemsep}{2pt}%
   \setlength{\parskip}{2pt}%
  }%
  {\end{oldthebibliography}}
\bibliography{refs}

\end{document}